\documentclass{amsart}
\usepackage{bc}
\usepackage{pgfplots}
\pgfplotsset{compat=1.12}
\newcommand\dd{\mathbf{d}}

\newcommand\phantombullet[1]{\makebox[0mm][r]{#1\phantom{$\bullet$}\quad}}
\newcommand\Sph{{\mathscr S}}
\newcommand\Ann{{\mathscr A}}
\newcommand\Fatou{{\mathcal F}}
\newcommand\Julia{{\mathcal J}}

\newenvironment{fsa}[1][auto]{\begin{tikzpicture}[->,>=stealth',
    shorten >=1pt,auto,node distance=3cm,double distance between line centers=0.45ex,
    initial text=,accepting/.style=accepting by arrow,
    every state/.style={inner sep=3pt,minimum size=0pt},
    every loop/.style={looseness=12},semithick,#1]}{\end{tikzpicture}}

\begin{document}
\title[Expanding maps]{Algorithmic aspects of branched coverings IV/V.\\ Expanding maps}
\author{Laurent Bartholdi}
\email{laurent.bartholdi@gmail.com}
\author{Dzmitry Dudko}
\email{dzmitry.dudko@gmail.com}
\address{\'Ecole Normale Sup\'erieure, Paris \emph{and} Mathematisches Institut, Georg-August Universit\"at zu G\"ottingen}
\thanks{Partially supported by ANR grant ANR-14-ACHN-0018-01 and DFG grant BA4197/6-1}
\date{October 7, 2016}
\begin{abstract}
  Thurston maps are branched self-coverings of the sphere whose
  critical points have finite forward orbits. We give combinatorial
  and algebraic characterizations of Thurston maps that are isotopic
  to expanding maps as \emph{Levy-free} maps and as maps with
  \emph{contracting biset}.

  We prove that every Thurston map decomposes along a unique minimal
  multicurve into Levy-free and finite-order pieces, and this
  decomposition is algorithmically computable. Each of these pieces
  admits a geometric structure.

  We apply these results to matings of post-critically finite
  polynomials, extending a criterion by Mary Rees and Tan Lei: they
  are expanding if and only if they do not admit a cycle of periodic
  rays.
\end{abstract}
\maketitle

\section{Introduction}
Let $f\colon(S^2,A)\selfmap$ be a branched covering of the sphere with
finite, forward-invariant set $A$ containing $f$'s critical values,
from now on called a \emph{Thurston map}. A celebrated theorem by
Thurston~\cite{douady-h:thurston} gives a topological criterion for
$f$ to be isotopic to a rational map, for an appropriate complex
structure on $(S^2,A)$. One of the virtues of rational maps, following
from Schwartz's lemma, is that they are expanding for the hyperbolic
metric of curvature $-1$ associated with the complex structure.

In this article, following the announcement
in~\cite{bartholdi-dudko:bc0}, we give a criterion for $f$ to be
isotopic to an expanding map, namely for there to exist a metric on
$(S^2,A)$ that is expanded by a map isotopic to $f$. It will turn out
that the metric may, for free, be required to be Riemannian of pinched
negative curvature.

Some care is needed to define expanding maps with periodic critical
points.  Consider a non-invertible map $f\colon (S^2,A)\selfmap$. Let
$A^\infty\subseteq A$ denote the forward orbit of the periodic
critical points of $f$. The map $f$ is \emph{metrically expanding} if
there exists a subset $A'\subseteq A^\infty$ and a metric on
$S^2\setminus A'$ that is expanded by $f$, and such that at all
$a\in A'$ the first return map of $f$ is locally conjugate to
$z\mapsto z^{\deg_a(f^n)}$. In other words, the points in $A'$ are
cusps, or equivalently at infinite distance, in the metric.

We call $f$ \emph{B\"ottcher expanding} if $A'=A^\infty$. This
definition is designed to generalize the class of rational
maps. Indeed, every post-critically finite rational map
$f\colon (\hC,A)\selfmap $ is B\"ottcher expanding by considering the
hyperbolic (or Euclidean if $|A|=2$) metric of $(\hC,\ord)$ for an
appropriate orbifold structure $\ord\colon A\to \N\cup \{\infty\}$.

We call $f$ \emph{topologically expanding} if there exists a compact
retract $\mathcal M\subset S^2\setminus A'$ and a finite open covering
$\mathcal M=\bigcup\mathcal U_i$ such that connected components of
$f^{-n}(\mathcal U_i)$ get arbitrarily small as $n\to\infty$ and such
that $S^2\setminus \mathcal M$ is in the immediate attracting basin of
$A'$; see~\cite{bartholdi-h-n:aglt}.  If $A^\infty=\emptyset$,
B\"ottcher expanding maps are the everywhere-expanding maps considered
e.g.\ in~\cites{haissinsky-pilgrim:cxc,bonk-meyer:etm}.

An obstruction to topological expansion is the existence of a
\emph{Levy cycle}. This is an essential simple closed curve on
$S^2\setminus A$ that is isotopic to some iterated preimage of
itself. We shall see that it is the only obstruction.

We recall briefly the algebraic encoding of branched coverings: given
$f\colon(S^2,A)\selfmap$, set $G\coloneqq \pi_1(S^2\setminus A,*)$,
and define
\[B(f)\coloneqq\{\gamma\colon[0,1]\to\mathcal M\mid\gamma(0)=f(\gamma(1))=*\}\,/\,\text{homotopy}.
\]
This is a set with commuting left and right $G$-actions,
see~\S\ref{ss:bisets} to which we refer for the definition of
\emph{contracting} bisets. Two branched self-coverings
$f_0\colon(S^2,A_0)\selfmap$ and $f_1\colon(S^2,A_1)\selfmap$ are
\emph{combinatorially equivalent} if there is a path
$(f_t\colon(S^2,A_t)\selfmap)_{t\in[0,1]}$ of branched self-coverings
joining them; this happens precisely when the bisets $B(f_0)$ and
$B(f_1)$ are isomorphic in a suitably defined sense,
see~\cite{kameyama:thurston} and~\cite{bartholdi-dudko:bc2}. The main
result of this part is the following criterion; equivalence
of~\eqref{main:topexp} and~\eqref{main:contracting} was known in the
case $A^\infty=\emptyset$
from~\cite{haissinsky-pilgrim:algebraic}*{Theorem~4}:
\begin{mainthm}[= Theorem~\ref{thm:ExpCr}]\label{thm:main}
  Let $f\colon(S^2,A)\selfmap$ be a Thurston map, not doubly covered
  by a torus endomorphism. The following are equivalent:
  \begin{enumerate}
  \item $f$ is combinatorially equivalent to a B\"ottcher metrically expanding map;\label{main:Bottcher}
  \item $f$ is combinatorially equivalent to a topologically expanding map;\label{main:topexp}
  \item $B(f)$ is an orbisphere contracting biset;\label{main:contracting}
  \item $f$ is non-invertible and admits no Levy cycle.
  \end{enumerate}

  Furthermore, if these properties hold, the metric
  in~\eqref{main:Bottcher} may be assumed to be Riemannian of pinched
  negative curvature.
\end{mainthm}

Ha\"\i ssinsky and Pilgrim ask in~\cite{haissinsky-pilgrim:algebraic}
whether every everywhere-expanding map is isotopic to a smooth map. By
Theorem~\ref{thm:main}, a combinatorial equivalence class contains a
B\"ottcher smooth expanding map if and only if it is Levy free. If
$A^\infty =\emptyset$, then a B\"ottcher expanding map is expanding
everywhere.

\subsection{Geometric maps and decidability}
Let us define \emph{\Tor} maps as self-maps of the sphere $S^2$ that
are a quotient of a torus endomorphism $M z+v\colon\R^2\selfmap$ by
the involution $z\mapsto-z$, such that the eigenvalues of $M$ are
different from $\pm 1$. Let us call a Thurston map \emph{geometric} if
it is either B\"ottcher expanding or \Tor.

Recall from~\cite{bartholdi-dudko:bc2} that $R(f,A,\CC)$ denotes the
small return maps of the decomposition of a Thurston map $f$ under an
invariant multicurve $\CC$. The \emph{canonical Levy obstruction}
$\CC_{\text{Levy}}$ of a Thurston map $f\colon(S^2,A)\selfmap$ is a
minimal $f$-invariant multicurve all of whose small Thurston maps are
either homeomorphisms or admit no Levy cycle. It is unique by
Proposition~\ref{prop:HypLevyMCurvInter}. The \emph{Levy
  decomposition} of $f$ (and equivalently of its biset) is its
decomposition (as a graph of bisets) along the canonical Levy
obstruction. It was proven
in~\cite{selinger-yampolsky:geometrization}*{Main Theorem~II} that
every Levy-free map that is doubly covered by a torus endomorphism is
in \Tor. Combined with Theorem~\ref{thm:main}, this implies
\begin{maincor}
  Let $f\colon(S^2,A)\selfmap$ be a Thurston map. Then every map in
  $R(f,A,\CC_{\text{Levy}})$ is either geometric or a homeomorphism.\qed
\end{maincor}

The following consequences are essential for the decidability of
combinatorial equivalence of Thurston maps.
\begin{maincor}[= Algorithms~\ref{algo:istor} and~\ref{algo:isexp}]\label{cor:decidegeom}
  There is an algorithm that, given a Thurston map by its biset,
  decides whether it is geometric.
\end{maincor}

\noindent As a consequence we have
\begin{maincor}[= Algorithm~\ref{algo:levy}]\label{cor:decidelevy}
  Let $f$ be a Thurston map. Then its Levy decomposition is
  symbolically computable.
\end{maincor}

There may exist expanding maps in the combinatorial equivalence class
of a Thurston map which are not B\"ottcher expanding. However, every
expanding map is a quotient of a B\"ottcher expanding map, by
Theorem~\ref{thm:main} combined with the following
\begin{prop}[= Proposition~\ref{prop:semiconjugate}]
  Let $f,g\colon(S^2,A)\selfmap$ be isotopic Thurston maps, let
  $\Fatou(f),\Fatou(g)$ be their respective Fatou sets
  (see~\S\ref{ss:fatou}), and assume
  $A\cap(\Fatou(g)\setminus\Fatou(f))=\emptyset$.  Then there is a
  semiconjugacy from $f$ to $g$, defined by collapsing to
  points those components of $\Fatou(f)$ that are attracted towards
  $A\cap(\Fatou(f)\setminus\Fatou(g))$ under $f$.
\end{prop}
We will show in~\cite{bartholdi-dudko:bc3} that the semiconjugacy is
unique.

We deduce the following extension to expanding maps of a classical
result for rational maps (see
e.g.~\cite{douady-h:thurston}*{Corollary~3.4(b)}) to B\"ottcher
expanding maps:
\begin{cor}[=Lemma~\ref{lem:ConjBetwExpMaps}]
  Let $f,g$ be B\"ottcher expanding Thurston maps. Then $f$ and $g$
  are combinatorially equivalent if and only if they are conjugate.\qed
\end{cor}

We also characterize maps (such as rational maps with Julia set a
Sierpi\'nski carpet) that are isotopic to an everywhere-expanding
map. A \emph{Levy arc} for a Thurston map $f\colon(S^2,A)\selfmap$ is
a non-trivial path with endpoints in $A$ which is isotopic to an
iterated lift of itself:
\begin{prop}[= Lemma~\ref{lem:HomIsolCond} with $A=A'$]
  Consider a Thurston map $f$ that is not doubly covered by a torus
  endomorphism. Then $f$ is isotopic to an everywhere-expanding map if
  and only if $f$ admits no Levy obstruction nor Levy arc.
\end{prop}

\subsection{Matings and amalgams}
We finally apply Theorem~\ref{thm:main} to the study of matings, and
more generally amalgams of expanding maps. We state the results for
matings in this introduction, while~\S\ref{ss:matings} will discuss
the general case of amalgams.

Let $p_+(z)=z^d+\cdots$ and $p_-(z)=z^d+\cdots$ be two post-critically
finite monic polynomials of same degree. Denote by $\overline\C$ the
compactification of $\C$ by a circle at infinity
$\{\infty\exp(2\pi i\theta)\}$, and consider the sphere
\[\mathbb S\coloneqq (\overline\C\times\{\pm1\})\,/\,\{(\infty\exp(2\pi i\theta),+1)\sim(\infty\exp(-2\pi i\theta),-1)\}.\]
(Note the reversed orientation between the two copies of
$\overline\C$). The \emph{formal mating}
\begin{equation}\label{eq:mating}
  p_+ \FM p_-\colon \mathbb S\selfmap,\quad (z,\varepsilon)\mapsto(p_\varepsilon(z),\varepsilon)
\end{equation}
is the branched covering of $\mathbb S$ acting as $p_+$ on its
northern hemisphere, as $p_-$ on its southern hemisphere, and as $z^d$
on the common equator $\{\infty\exp(2i\pi\theta)\}$. The maps
$p_+,p_-$ glue continuously by Lemma~\ref{lem:Bottcher extension}.

We recall the definition of \emph{external rays} associated to the
polynomials $p_\pm$. For a polynomial $p$, the \emph{filled-in Julia
  set} $K_p$ of $p$ is
\[K_p=\{z\in\C\mid p^n(z)\not\to\infty\text{ as }n\to\infty\}.\]
Assume that $K_p$ is connected. There exists then a unique holomorphic
isomorphism
$\phi_p\colon\hC\setminus K_p\to\hC\setminus\overline{\mathbb D}$
satisfying $\phi_p(p(z))=\phi_p(z)^d$ and $\phi_p(\infty)=\infty$ and
$\phi_p'(\infty)=1$. It is called a \emph{B\"ottcher coordinate}, and
conjugates $p$ to $z^d$ in a neighbourhood of $\infty$. For
$\theta\in\R/\Z$, the associated \emph{external ray} is
\[R_p(\theta)=\{\phi_p^{-1}(re^{2i\pi\theta})\mid r\ge1\}.
\]
Let $\Sigma$ denote the quotient of $\mathbb S$ in which each
$\overline{(R_{p_\varepsilon}(\theta),\varepsilon)}$ has been
identified to one point for each $\theta\in\R/\Z$ and each
$\varepsilon\in\{\pm1\}$. Note that $\Sigma$ is a quotient of
$K_{p_+}\sqcup K_{p_-}$, and need not be a Hausdorff space. A
classical criterion (due to Moore) determines when $\Sigma$ is
homeomorphic to $S^2$. If this occurs, $p_+$ and $p_-$ are said to be
\emph{topologically mateable}, and the map induced by $p_+\FM p_-$ on
$\Sigma$ is called the \emph{topological mating} of $p_+$ and
$p_-$ and denoted $p_+\GM p_-\colon \Sigma\selfmap$.

\begin{defn}
  Let $p_+,p_-$ be two monic post-critically finite polynomials of
  same degree $d$. We say that $p_+,p_-$ have a \emph{pinching cycle
    of periodic angles} if there are angles
  $\phi_0,\phi_1, \dots, \phi_{2n-1}\in \Q/\Z$ with denominators
  coprime to $d$, such that for all $\varepsilon=\pm1$ and all
  $i=0,\dots,2n-1$, indices treated modulo $2n$, the rays
  $R_{p_\varepsilon}( \varepsilon \phi_{2i})$ and
  $R_{p_\varepsilon}(\varepsilon \phi_{2i+\varepsilon})$ land
  together.
\end{defn}

We give a computable criterion for two hyperbolic polynomials to be
mateable, which extends a well-known criterion ``two quadratic
polynomials are geometrically mateable if and only if they do not
belong to conjugate primary limbs in the Mandelbrot set'' due to Mary
Rees and Tan Lei, see~\cite{tan:matings}
and~\cite{buff+:questionsaboutmatings}*{Theorem~2.1}:
\begin{mainthm}\label{thm:mating}
  Let $p_+,p_-$ be two monic hyperbolic post-critically finite polynomials. Then
  the following are equivalent:
  \begin{enumerate}
  \item\label{thm:mating:1} $p_+ \FM p_-$ is combinatorially
    equivalent to an expanding map;
  \item\label{thm:mating:2} $p_+ \GM p_-$ is a sphere map (necessarily
    conjugate to any expanding map in~\eqref{thm:mating:1});
  \item\label{thm:mating:3} $p_+,p_-$ do not have a pinching cycle of
    periodic angles.
  \end{enumerate}
\end{mainthm}

To be more precise, the criterion due to Mary Rees and Tan Lei relies
on the fact that, in degree $2$, every Thurston obstruction is a Levy
obstruction, so every expanding map is automatically conjugate to a
rational map. In degree $\ge3$ there are topological matings that are
not conjugate to rational maps: the example
in~\cite{shishikura-t:matings} is precisely such a mating with an
obstruction but no Levy obstruction, and it is isotopic to an
expanding map. 

Furthermore, in degree $2$ every decomposition of a Thurston map along
a Levy cycle has a fixed sphere or cylinder which maps to itself by a
homeomorphism cyclically permuting the boundary components (namely,
there exists a ``good Levy cycle''). This implies that obstructed maps
have a pinching cycle of periodic angles with $n=2$. In
Example~\ref{exple:degree3}, we show that this does not hold in higher
degree.

\subsection{Notation}
Let $f\colon(S^2, A)\selfmap$ be a Thurston map with an invariant
multicurve $\CC$. Recall that by $R(f,A,\CC)$ we denote the return maps
induced by $f$ on $S^2\setminus\CC$,
see~\cite{bartholdi-dudko:bc2}*{\S\ref{bc2:ss:return maps}}.

We introduce the following notation. By default, curves and
multicurves are considered up to isotopy rel the marked points; we use
the terminology ``equal'' to mean that. In particular, a cycle of
curves is really a sequence of curves that are mapped cyclically to
each other, up to isotopy. If we want to insist that curves are equal
and not just isotopic, we add the adjective ``solid''; thus a solid
cycle of curves is a a sequence of curves mapped cyclically to each
other, ``on the nose''.

We reserve letters `$\CC$' for invariant multicurves and `$C$' for
cycles of curves, or more generally for subsets of invariant
multicurves.

\section{Multicurves and the Levy decomposition}
Let $A$ be a finite subset of the topological sphere $S^2$, and
consider simple closed curves on $S^2\setminus A$. Recall that such a
curve is \emph{trivial} if it bounds a disc in $S^2\setminus A$, and
is \emph{peripheral} if it may be homotoped into arbitrarily small
neighbourhoods of $A$; otherwise, it is \emph{essential}. A
\emph{multicurve} is a collection of mutually non-intersecting
non-homotopic essential simple closed curves. Following
Harvey~\cite{harvey:curvecomplex}, we denote by
$\mathcal C(S^2\setminus A)$ the flag complex whose vertices are
isotopy classes of essential curves, and a collection of curves belong
to a simplex if they have disjoint representatives; so multicurves on
$S^2\setminus A$ are naturally identified with simplices in
$\mathcal C(S^2\setminus A)$. (The empty multicurve corresponds to the
empty simplex).

Given two simple closed curves $\gamma_1$ and $\gamma_2$ on
$S^2\setminus A$, their \emph{geometric intersection number} is
defined as
\[i(\gamma_1,\gamma_2)=\min_{\gamma'_1,\gamma'_2}\#(\gamma'_1\cap\gamma'_2),
\]
with the minimum ranging over all curves $\gamma'_1$ isotopic to
$\gamma_1$ and $\gamma'_2$ isotopic to $\gamma_2$.  The simple closed
curves $\gamma_1$ and $\gamma_2$ are in \emph{minimal position} if
$i(\gamma_1,\gamma_2)=\#(\gamma_1\cap \gamma_2)$.

We say that two simple closed curves $\gamma_1$ and $\gamma_2$
\emph{cross} if $i(\gamma_1,\gamma_2)>0$. Clearly, if $\gamma_1$ and
$\gamma_2$ are isotopic or one of them is inessential, then
$i(\gamma_1,\gamma_2)=0$. Two multicurves $\CC_1$ and $\CC_2$
\emph{cross} if there are $\gamma_1\in \CC_1$ and $\gamma_2\in \CC_2$
that cross.

\begin{prop}[The Bigon Criterion, \cite{farb-margalit:mcg}*{Proposition~1.3}]\label{prop:BigonCriterion}
  Two transverse simple closed curves on a surface $S$ are in minimal
  position if and only if the two arcs between any pair of
  intersection points never bound an embedded disc in $S$.\qed
\end{prop}

\subsection{Levy, anti-Levy, Cantor, and anti-Cantor multicurves}
Consider a Thurston map $f\colon(S^2,A)\selfmap$. We construct the
following directed graph: its vertex set is the set of essential
simple closed curves on $S^2\setminus A$, namely the vertex set of the
curve complex $\mathcal C(S^2\setminus A)$. For every simple closed
curve $\gamma$ and for every component $\delta$ of $f^{-1}(\gamma)$,
we put an edge from $\gamma$ to $\delta$ labeled
$\deg(f\restrict\delta)$. Note that the operation $f^{-1}$ induces a
map on the simplices of $\mathcal C(S^2\setminus A)$, but not a
simplicial map.

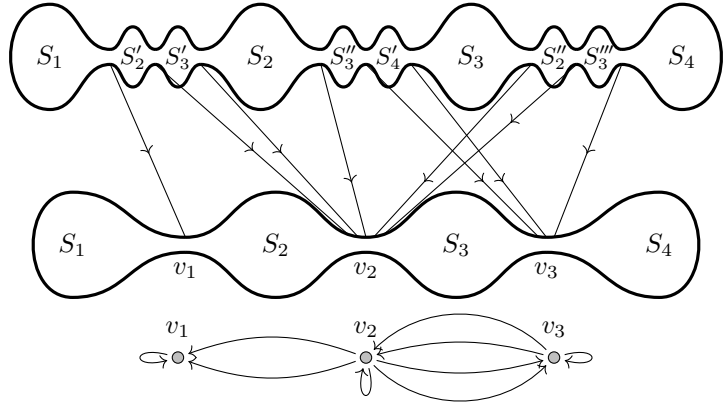
\begin{figure}
  \begin{tikzpicture}[auto,vx/.style={circle,minimum size=1ex,inner sep=0pt,outer sep=2pt,draw,fill=gray!50}]
    \def\leftsphere#1#2{+(0,0.1) .. controls +(180:#1/2) and +(0:#1/2) .. +(-#1,#2)
      .. controls +(180:#2) and +(180:#2) .. +(-#1,-#2)
      .. controls +(0:#1/2) and +(180:#1/2) .. +(0,-0.1) ++(0,0)}
    \def\rightsphere#1#2{+(0,0.1) .. controls +(0:#1/2) and +(180:#1/2) .. +(#1,#2)
      .. controls +(0:#2) and +(0:#2) .. +(#1,-#2)
      .. controls +(180:#1/2) and +(0:#1/2) .. +(0,-0.1) ++(0,0)}
    \def\midsphere#1#2{+(0,0.1) .. controls +(0:#1/2) and +(180:#1/2) .. +(#1,#2)
      .. controls +(0:#1/2) and +(180:#1/2) .. +(#1*2,0.1)
      +(0,-0.1) .. controls +(0:#1/2) and +(180:#1/2) .. +(#1,-#2)
      .. controls +(0:#1/2) and +(180:#1/2) .. +(#1*2,-0.1) ++(#1*2,0)
    }
    
    \draw[very thick] (-2.4,0) node [xshift=-15mm] {$S_1$}
    \leftsphere{1.5}{0.7} node [below=1mm] {$v_1$} node[xshift=12mm] {$S_2$}
    \midsphere{1.2}{0.7} node[below=1mm] {$v_2$} node[xshift=12mm] {$S_3$}
    \midsphere{1.2}{0.7} node[below=1mm] {$v_3$} node[xshift=15mm] {$S_4$}
    \rightsphere{1.5}{0.7};

    \draw[very thick] (-3.4,2.5) node [xshift=-8mm] {$S_1$}
    \leftsphere{0.8}{0.7} node [xshift=3mm] {\small $S'_2$}
    \midsphere{0.3}{0.4} node [xshift=3mm] {\small $S'_3$}
    \midsphere{0.3}{0.4} node [xshift=8mm] {$S_2$} \midsphere{0.8}{0.7}
    node [xshift=3mm] {\small $S''_3$} \midsphere{0.3}{0.4} node
    [xshift=3mm] {\small $S'_4$} \midsphere{0.3}{0.4} node [xshift=8mm]
    {$S_3$} \midsphere{0.8}{0.7} node [xshift=3mm] {\small $S''_2$}
    \midsphere{0.3}{0.4} node [xshift=3mm] {\small $S'''_3$}
    \midsphere{0.3}{0.4} node [xshift=8mm] {$S_4$} \rightsphere{0.8}{0.7};

    \draw (-3.4,2.4) -- (-2.4,0.1) \fwdarrowonline{0.5};
    \draw (-2.8,2.4) -- (-0.25,0.15) \fwdarrowonline{0.5};
    \draw (-2.2,2.4) -- (-0.1,0.1) \fwdarrowonline{0.5};
    \draw (-0.6,2.4) -- (0.0,0.1) \fwdarrowonline{0.667};
    \draw (0.0,2.4) -- (2.3,0.1) \fwdarrowonline{0.667};
    \draw (0.6,2.4) -- (2.4,0.1) \fwdarrowonline{0.667};
    \draw (2.2,2.4) -- (0.1,0.1) \bckarrowonline{0.667};
    \draw (2.8,2.4) -- (0.25,0.15) \bckarrowonline{0.333};
    \draw (3.4,2.4) -- (2.5,0.1) \bckarrowonline{0.5};

    \node[vx,label={$v_1$}] (v1) at (-2.5,-1.5) {};
    \node[vx,label={$v_2$}] (v2) at (0,-1.5) {};
    \node[vx,label={$v_3$}] (v3) at (2.5,-1.5) {};
    \draw[<-] (v1)  to [loop left] (v1);
    \draw[<-] (v1)  to [bend left=20] (v2);
    \draw[<-] (v1)  to [bend right=20] (v2);
    \draw[->] (v2)  to [loop below] (v2);
    \draw[->] (v2)  to [bend right=15] (v3);
    \draw[->] (v2)  to [bend right=45] (v3);
    \draw[->] (v3)  to [loop right] (v3);
    \draw[->] (v3)  to [bend right=15] (v2);
    \draw[->] (v3)  to [bend right=45] (v2);
  \end{tikzpicture}
  \caption{A bicycle $\{v_2,v_3\}$ generates a Cantor multicurve
    $\{v_1,v_2,v_3\}$. The action of the map $f$ is indicated on the
    preimages of $\{v_1,v_2,v_3\}$. If annuli are mapped by degree
    $1$, then it is also a Levy cycle. Trivial spheres are omitted on
    the top sphere. The graph below is the corresponding portion of
    the graph on $\mathcal C(S^2\setminus
    A)$.} \label{Fig:ExampleCantorMult}
\end{figure}

A multicurve $\CC\in\mathcal C(S^2\setminus A)$ is \emph{invariant}
if $f^{-1}(\CC)=\CC$. Given a multicurve $\CC_0$ with
$\CC_0\subseteq f^{-1}(\CC_0)$, there is a unique invariant multicurve
$\CC$ \emph{generated} by $\CC_0$, namely the intersection of all
invariant multicurves containing $\CC_0$. The invariant multicurve
$\CC$ may readily be computed by considering
$\CC_0,f^{-1}(\CC_0),f^{-2}(\CC_0),\dots$; this is an ascending
sequence of multicurves, and each multicurve contains at most $\#A-3$
curves so the sequence must stabilize.

Let $\CC$ be an invariant multicurve, and consider the corresponding
directed subgraph of $\mathcal C(S^2\setminus A)$. A \emph{strongly
  connected component} is a maximal subgraph spanned by a subset
$C\subseteq\CC$ such that, for every $\gamma,\delta\in C$, there
exists a non-trivial path from $\gamma$ to $\delta$ in $C$. Note that
singletons with no loop are never strongly connected components.

Strongly connected components are partially ordered: $C\prec D$ if
there is a path from a curve in $C$ to a curve in $D$. Consider a
strongly connected component $C$. We call $C$ \emph{primitive in
  $\CC$} if it is minimal for $\prec$. We call $C$ a \emph{bicycle} if
for every $\gamma,\delta\in C$ there exists $n\in\N$ such that at
least two paths of length $n$ join $\gamma$ to $\delta$ in $C$, and a
\emph{unicycle} otherwise; see Figure~\ref{Fig:ExampleCantorMult} for
an illustration.

We remark that bicycles contain at least two cycles, so the number of
paths of length $n$ grows exponentially in $n$. On the other hand,
every unicycle is an actual \emph{periodic cycle}, namely can be
written as $C=(\gamma_0,\gamma_1,\dots,\gamma_n=\gamma_0)$ in such a
manner that $\gamma_{i+1}$ has an $f$-preimage $\gamma'_i$ isotopic to
$\gamma_i$. If in a periodic cycle $C$ the $\gamma'_i$ may be chosen
so that $f$ maps each $\gamma'_i$ to $\gamma_{i+1}$ by degree $1$,
then $C$ is called a \emph{Levy cycle}.

A periodic cycle $C=(\gamma_0,\gamma_1,\dots,\gamma_n=\gamma_0)$ is a
\emph{solid periodic cycle} if $f$ maps $\gamma_i$ onto $\gamma_{i+1}$
for all $i=0,\dots,n-1$; if $f$ maps every $\gamma_i$ to
$\gamma_{i+1}$ by degree $1$, then $C$ is called a \emph{solid Levy
  cycle}. Since the critical values of $f$ are assumed to belong to
$A$, the restrictions
$f\restrict{\gamma_i}\colon\gamma_i\to\gamma_{i+1}$ are all
homeomorphisms. Note that a periodic cycle may be isotopic to more
than one solid periodic cycle, possibly some solid Levy and some solid
non-Levy cycles.

\[\begin{tikzpicture}[->,auto,vx/.style={circle,minimum size=1ex,inner sep=0pt,outer sep=2pt,draw,fill=gray!50}]
  \node[vx] (1) at (-1,0) {};
  \node[vx] (2) at (1,0) {};
  \draw (1)  to [bend left=60] (2);
  \draw (1)  to [bend left=10] (2);
  \draw (2)  to [bend left] (1);
  \node at (0,-0.6) {bicycle};
\end{tikzpicture}
\qquad
\begin{tikzpicture}[->,auto,vx/.style={circle,minimum size=1ex,inner sep=0pt,outer sep=2pt,draw,fill=gray!50}]
  \node[vx] (1) at (-1,0) {};
  \node[vx] (2) at (0,1.5) {};
  \node[vx] (3) at (1,0) {};
  \draw (1)  to [bend left=10] node {$1:1$} (2);
  \draw (2)  to [bend left=10] node {$1:1$} (3);
  \draw (3)  to [bend left=10] node [above] {$1:1$} (1);
  \node at (0,-0.6) {Levy cycle};
\end{tikzpicture}
\qquad
\begin{tikzpicture}[->,auto,vx/.style={circle,minimum size=1ex,inner sep=0pt,outer sep=2pt,draw,fill=gray!50}]
  \node[vx] (1) at (-1,0) {};
  \node[vx] (2) at (-1,1) {};
  \node[vx] (3) at (1,0) {};
  \node[vx] (4) at (1,1) {};
  \draw[dashed,thick] (-1,0.5) ellipse (6mm and 9mm) node[above right=7mm] {$C$};
  \draw (1)  to [bend left] (2);
  \draw (2)  to [bend left] (1);
  \draw (1)  to (3);
  \draw (3)  to [bend left] (4);
  \draw (4)  to [bend left] (3);
  \node at (0,-0.6) {Primitive s.c.c.};
\end{tikzpicture}
\]

We remark that every invariant multicurve is generated by its
primitive unicycles and bicycles, and that if $C$ is a strongly
connected component of an invariant multicurve $\CC$ and $C$ has a
curve in common with an invariant multicurve $\DD$ then $C$ is also a
strongly connected component in $\DD$; and it is a bicycle in $\CC$ if
and only if it is a bicycle in $\DD$. However, $C$ could be primitive
in $\CC$ but not in $\DD$.

We will sometimes speak of a strongly connected component without
reference to an invaraint multicurve containing it. We will also say
that a strongly connected component $C$ is \emph{primitive} if it is
primitive in any invariant multicurve containing $C$.

\begin{defn}[Types of invariant multicurves]
  Let $\CC$ be an invariant multicurve. Then $\CC$ is
  \begin{idescription}
  \item[Cantor] if it is generated by its bicycles;
  \item[anti-Cantor] if $\CC$ does not contain any bicycle;
  \item[Levy] if it is generated by its Levy cycles;
  \item[anti-Levy] if $\CC$ does not contain any Levy cycle.\qedhere
  \end{idescription}
\end{defn}

\begin{prop}\label{Prop:CantorMultCurv}
  Suppose $f\colon(S^2, A)\selfmap$ is a Thurston map with an
  invariant multicurve $\CC$. Then
  \begin{enumerate}
  \item there is a unique maximal invariant Cantor sub-multicurve
    $\CC_{\text{Cantor}}\subseteq\CC$ such that the restrictions of
    $\CC$ to pieces in $S^2\setminus\CC_{\text{Cantor}}$ are
    anti-Cantor invariant multicurves of return maps in
    $R(f,A,\CC_{\text{Cantor}})$;
  \item there is a unique maximal invariant Levy sub-multicurve
    $\CC_{\text{Levy}}\subseteq\CC$ such that the restrictions of
    $\CC$ to pieces in $S^2\setminus\CC_{\text{Levy}}$ are anti-Levy
    invariant multicurves of return maps in
    $R(f,A,\CC_{\text{Levy}})$.
\end{enumerate}
\end{prop}
\begin{proof}
  The multicurve $\CC_{\text{Cantor}}$ is generated by all the
  bicycles in $\CC$ while the multicurve $\CC_{\text{Levy}}$ is
  generated by all the Levy cycles in $\CC$.
\end{proof}

\subsection{Crossings of Levy cycles}
We now show that Levy cycles cross invariant multicurves in a quite
restricted way. First we need the following technical properties.

\begin{prop}\label{prop:solidPerCycl}
  Let $f\colon(S^2,A)\selfmap$ be a Thurston map. Then
  \begin{enumerate}
  \item if $C$ is a periodic cycle, then there is a homeomorphism
    $h\colon(S^2,A)\selfmap$ isotopic to the identity rel $A$ such
    that $C$ is a solid periodic cycle of the map $h\circ f$, and is
    Levy for $h\circ f$ if it was Levy for
    $f$;\label{prop:solidPerCycl:1}
  \item if a periodic cycle $C$ crosses a Levy cycle, then $C$ is a
    periodic primitive unicycle. A strictly preperiodic curve does not
    cross a Levy cycle;\label{prop:solidPerCycl:2}
  \item if $L$ is a Levy cycle and $C$ is a periodic cycle crossing
    $L$ such that $C$ and $L$ are in minimal position, then there is
    homeomorphism $h\colon(S^2,A)\selfmap$ isotopic to the identity
    rel $A$ such that $C$ and $L$ are solid curve cycles of the map
    $h\circ f$.\label{prop:solidPerCycl:3}
  \end{enumerate}
\end{prop}

We remark that the last statement can not be improved much. Indeed,
there is an example, due to Wittner~\cite{wittner:phd}, of the mating
of the airplane and rabbit polynomials, which may be decomposed in two
manners as a mating; in other words, the map admits two ``equators''
(invariant curves mapped $d:1$ to themselves). It is impossible to
make both equators simultaneously solidly periodic and in minimal
position; worse, if they are both made solidly periodic, then they
must have infinitely many crossings. We recall the following
\begin{lem}[The Alexander method, \cite{farb-margalit:mcg}*{Proposition~2.8}]\label{lem:alexander}
  A collection of pairwise non-isotopic essential curves
  $\{\gamma_i\}_i$ can be simultaneously isotopically moved into
  $\{\gamma'_i\}_i$ if (1) all curves in $\{\gamma_i\}_i$ are pairwise
  in minimal position, (2) all curves in $\{\gamma'_i\}_i$ are
  pairwise in minimal position, (3) every $\gamma_i$ is isotopic to
  the corresponding $\gamma'_i$, and (4) for pairwise different
  $i,j,k$ at least one of $i(\gamma_i,\gamma_j)$,
  $i(\gamma_i,\gamma_k)$ and $i(\gamma_j,\gamma_k)$ is $0$.\qed
\end{lem}
\begin{proof}[Proof of Proposition~\ref{prop:solidPerCycl}]
  We begin by~\eqref{prop:solidPerCycl:1}. Write
  $C=(\gamma_0,\gamma_1,\dots,\gamma_n=\gamma_0)$. For every $i$
  choose a component $\gamma'_i$ of $f^{-1}(\gamma_{i+1})$ that is
  isotopic to $\gamma_i$, mapping by degree $1$ if $C$ is a Levy
  cycle. Note that the $\gamma'_i$ are disjoint. Any isotopy moving
  all $\gamma_i$ to $\gamma'_i$ satisfies the claim.

  Let us move to the second claim. Assume that
  $C=(\gamma_0,\gamma_1,\dots,\gamma_n=\gamma_0)$ crosses a Levy cycle
  $L$. By Part~\eqref{prop:solidPerCycl:1} we may assume that $L$
  is a solid Levy cycle.

  Put $\gamma_0$ in minimal position with respect to $L$ and denote
  by $\#(\gamma_0\cap L)$ the total number of crossings of
  $\gamma_0$ with $L$. Since $L$ is a solid Levy cycle we have
  \[\#(f^{-m}(\gamma_0)\cap L)=\#(\gamma_0\cap L)
  \]
  for every $m\ge 0$. If $m$ is a multiple of $n$, then
  $f^{-m}(\gamma_0)$ contains at least one component $\gamma'_0$
  isotopic to $\gamma_0$. By minimality,
  \[\#(\gamma'_0\cap L)\ge \#(\gamma_0\cap L).
  \]
  We conclude that for every $m\ge0$ there is exactly one component in
  $f^{-m}(\gamma_0)$ that crosses $L$. This component is necessarily
  isotopic to $\gamma_{-m}$, subscripts computed modulo
  $n$. Claim~\eqref{prop:solidPerCycl:2} follows from the observation
  that if $\gamma$ crosses a Levy cycle $L$, is periodic and is a
  preimage of some $\gamma'$, then $\gamma'$ crosses $L$.

  Let us prove Claim~\eqref{prop:solidPerCycl:3}. Write
  $L=(\lambda_0,\dots,\lambda_p=\lambda_0)$ and
  $C=(\gamma_0,\dots, \gamma_q=\gamma_0)$. By
  Part~\eqref{prop:solidPerCycl:1} we may assume that $L$ is a solid
  Levy cycle. By Part~\eqref{prop:solidPerCycl:2}, there is a unique
  component $\gamma'_i$ of $f^{-1}(\gamma_{i+1})$ that is isotopic to
  $\gamma_i$.  It follows from the above discussion that
  \[C'=(\gamma'_0,\gamma'_1,\dots, \gamma'_q)
  \]
  is also in minimal position with respect to $C$. It follows from the
  Alexander method, Lemma~\ref{lem:alexander}, that there is an
  isotopy moving every $\gamma'_i$ into $\gamma_i$ while fixing every
  $\lambda_i$.
\end{proof}

Let $f\colon(S^2,A)\selfmap$ be a Thurston map, and let $\CC$ be an
invariant multicurve. The components of $S^2\setminus\CC$ can be
compactified to \emph{small spheres} by shrinking each boundary
component to a point, and $f$ induces \emph{small maps} between the
small spheres, well defined up to isotopy. A periodic small sphere
$S_0$ gives rise to a \emph{small Thurston cycle} of maps
$S_0\to S_1\to\cdots\to S_0$ (see~\cite{bartholdi-dudko:bc2}*{Definition~\ref{bc2:lem:SmallMaps}}), which is a \emph{small homeomorphism
  cycle} if all the small maps are homeomorphisms.

The next result states that two Levy cycles can be joined so as to
give a finer decomposition, with additional homeomorphism small
maps. Its content is non-trivial only if the Levy cycles intersect.
\begin{cor}\label{cor:IntOfLevyCycl}
  Let $C_1$ and $C_2$ be two Levy cycles. Then a small neighbourhood
  of their union is a small homeomorphism cycle.

  More precisely, assume that $C_1$ and $C_2$ are in minimal
  position. Denote by $\CC$ the invariant multicurve generated by the
  boundary of a small neighbourhood of $C_1\cup C_2$ in $S^2$. Then
  the small spheres of $(S^2,A)\setminus\CC$ that intersect
  $C_1\cup C_2$ form a small homeomorphism cycle.
\end{cor}
\begin{proof}
  By Proposition~\ref{prop:solidPerCycl}\eqref{prop:solidPerCycl:3} we
  may assume that $C_1$ and $C_2$ are solid Levy cycles in minimal
  position.

  Let $\CC_0$ be the boundary of a small neighbourhood $N$ of
  $C_1\cup C_2$ in $S^2$. By the Bigon criterion,
  Proposition~\ref{prop:BigonCriterion}, all curves in $\CC_0$ are
  non-trivial.  For every $\gamma\in\CC_0$, its image $f(\gamma)$
  belongs to $\CC_0$ and the restriction
  $f\restrict\gamma\colon\gamma\to f(\gamma)$ has degree $1$. Since
  $f$ is a covering away from $A$, it extends to a homeomorphism on
  $N$. Up to isotopy, we may suppose that $N$ is invariant.

  Since $\CC$ does not contain the peripheral or trivial curves in
  $\CC_0$, we should extend $f\colon N\to N$ to all connected
  components of $S^2\setminus N$ that contain at most one marked
  point.

  By passing to an iterate of $f$ to lighten notation, we may assume
  that $f$ preserves each component of $\partial N$. Let $D$ be a
  disc in $S^2\setminus N$, and assume that $f\colon D\to f(D)$ has
  degree at least $2$. Since $f$ preserves $\partial D$ and is a
  homeomorphism on $N$, the image $f(D)$ contains $D$. Likewise,
  $f^{-1}(D)\cap D$ contains a component, say $E$, whose boundary
  contains $\partial D$. We get a map $f\colon E\to D$ of degree at
  least $2$; so $D$ contains at least two critical values, so it is
  essential.

  It follows that $f$ extends to a homeomorphism on the union of $N$
  and the inessential discs touching it.
\end{proof}

\subsection{The Levy decomposition}
A Thurston map $f\colon (S^2,A)\selfmap$ is called \emph{Levy-free} if
$f$ does not admit a Levy cycle and the degree of $f$ is at least
$2$. Here we characterize the multicurves along which $f$ decomposes
into Levy-free maps.

We say that an invariant Levy multicurve $\CC$ is \emph{complete} if
every small Thurston map in $R(f,A,\CC)$ either Levy-free or a
homeomorphism.

\begin{prop}\label{prop:HypLevyMCurvInter}
  Let $\CC_1$ and $\CC_2$ be complete invariant Levy multicurves of a
  Thurston map $f\colon(S^2,A)\selfmap$. Then
  \begin{enumerate}
  \item if a periodic curve $\gamma_1\in\CC_1$ crosses a curve
    $\gamma_2\in \CC_2$, then $\gamma_1$ and $\gamma_2$ belong to
    primitive Levy unicycles;\label{prop:HypLevyMCurvInter:1}
  \item the Levy-free maps in $R(f,A,\CC_1)$ and in $R(f,\CC_2)$ are the
    same;\label{prop:HypLevyMCurvInter:2}
  \item the multicurve $\CC_1\cap\CC_2$ is a complete invariant Levy
    multicurve.\label{prop:HypLevyMCurvInter:3}
  \end{enumerate}
\end{prop}

It follows that there is a unique minimal complete invariant Levy
multicurve, called the \emph{canonical Levy obstruction of $f$} and
written $\CC_{f,\text{Levy}}$. Any other invariant complete Levy
multicurve $\CC$ contains $\CC_{f,\text{Levy}}$ as a sub-multicurve.

\begin{proof}[Proof of Proposition~\ref{prop:HypLevyMCurvInter}]
  \eqref{prop:HypLevyMCurvInter:1} By the definition of a Levy
  multicurve, for every $\gamma_2\in\CC_2$ there is a Levy cycle $C_2$
  such that $\gamma_2$ is an iterated preimage of a curve in
  $C_2$. Consider $\gamma_1\in\CC_1$. Then $\gamma_1$ crosses $C_2$,
  because $\gamma_1$ is periodic. It follows from
  Proposition~\ref{prop:solidPerCycl}\eqref{prop:solidPerCycl:2} that
  $\gamma_1$ belongs to a primitive Levy unicycles, and by
  symmetry the same is true for $\gamma_2$.

  \eqref{prop:HypLevyMCurvInter:2} Consider a Levy-free cycle
  $f^p\colon S_0\to S_1\to\dots\to S_p=S_0$ in $R(f,A,\CC_1)$. We show
  that $\CC_2$ intersects none of the $S_1,S_2,\dots,S_p$; this
  implies that $\bigsqcup_i S_i$ is contained in a Levy-free cycle
  $S'_0\to\dots\to S'_{p'}=S'_0$ of small spheres of $R(f,A,\CC_2)$, and
  symmetrically $\bigsqcup_j S'_j$ is contained in a Levy-free cycle of
  small spheres of $R(f,A,\CC_1)$, so $\bigsqcup_i S_i$ and
  $\bigsqcup_j S'_j$ are the same.

  Assume therefore for contradiction that $\CC_2$ intersects some
  small sphere $S_i$. If this intersection is entirely contained in
  $S_i$, it will generate a Levy cycle in $\bigsqcup_i S_i$,
  contradicting the assumption that $\bigsqcup_i S_i$ is Levy-free;
  therefore $\CC_2$ crosses $\bigsqcup_i\partial S_i$.

  There is then a periodic curve in $\bigsqcup_i \partial S_i$
  crossing $\CC_2$. Choose a curve cycle
  $C_1\subseteq\bigsqcup_i \partial S_i$ and a curve cycle
  $C_2\subseteq\CC_2$ such that $C_1$ crosses $C_2$. By
  Part~\eqref{prop:HypLevyMCurvInter:1} of the proposition, $C_1$ and
  $C_2$ are anti-Cantor Levy cycles. By
  Corollary~\ref{cor:IntOfLevyCycl}, there is a small homeomorphism
  cycle $\{S'_i\}_i$ containing $C_1\cup C_2$. All curves in
  $\bigsqcup_i\partial S'_i$ belong to Levy cycles.

  If $\bigcup_i S'_i$ contains (up to isotopy) the union $\bigcup_i
  S_i$, then we have a contradiction because the degree of $f$ on
  $\bigcup_i S_i$ is at least $2$ while it is $1$ on $\bigcup_i
  S'_i$.

  We now show that we can always reduce to this case.  If
  $\bigcup_i S'_i$ does not contain $\bigcup_i S_i$, then then there
  is a curve cycle in $\bigsqcup_i\partial S'_i$ crossing at least one
  curve in $\bigsqcup_i \partial S_i$.  This implies that there is a
  Levy cycle in $\bigsqcup_i\partial S'_i$ crossing a Levy cycle in
  $\bigsqcup_i \partial S_i$. Invoking again
  Corollary~\ref{cor:IntOfLevyCycl}, we can enlarge $\bigcup_i
  S'_i$. Repeating, we may enlarge $\bigcup_i S'_i$ so that it
  contains $\bigcup_i S_i$.

  Finally, \eqref{prop:HypLevyMCurvInter:3} follows formally
  from~\eqref{prop:HypLevyMCurvInter:2}.
\end{proof}

\begin{defn}[Levy decomposition]\label{def:HypDecomp}
  The Levy decomposition of a Thurston map $f\colon(S^2,A)\selfmap$ is
  the decomposition of $f$ along the canonical Levy obstruction
  $\CC_{f,\text{Levy}}$.
\end{defn}

We may understand the Levy decomposition of a Thurston map
$f\colon(S^2,A)\selfmap$ as follows, if we consider more general
subsets of $S^2$ on which $f$ acts as a homeomorphism. Let us call
``Levy kernel'' a subset $L\subseteq S^2$ together with a partition
$L=\bigsqcup_{i\in I} S_i$ and a map $f\colon I\selfmap$ such that
each $S_i$ is either an essential simple closed curve or an essential
small sphere, and is considered up to isotopy; we require that every
$S_i$ be isotopic to a degree-$1$ preimage of $S_{f(i)}$, and that if
$S_i$ is a curve, then it is not homotopic to any (boundary) curve in
$\bigcup_{j\not=i}\partial S_j$. (The last condition replaces the
``non-homotopic'' condition in the definition of a multicurve.)  There
is a natural order on Levy kernels, given by inclusion up to isotopy.

We may think about a Levy kernel as a subset of a sphere on which $f$
has degree one.  Corollary~\ref{cor:IntOfLevyCycl} states that if two
Levy kernels $L_1$, $L_2$ intersect, then we can construct a bigger
Levy kernel $\widetilde L$ that contains both $L_1$ and
$L_2$. Therefore, there exists a maximal Levy kernel, and its boundary
generates the Levy decomposition.

\section{Self-similar groups and automata}\label{ss:bisets}
\noindent We recall basic notions about self-similar groups; the
reference is~\cite{nekrashevych:ssg}.

\subsection{Contracting bisets}
Let $G$ be a group. Recall that a \emph{$G$-$G$-biset} is a set $B$
endowed with commuting left and right $G$-actions. Such a biset $B$ is
called \emph{left-free} if the left $G$-action is free, i.e.\ has
trivial stabilizers. A \emph{basis} is a choice of one element per
left $G$-orbit: a subset $X\subseteq B$ such that
$B=\bigsqcup_{x\in X}G x$. We therefore have a bijection
$G\times X\leftrightarrow B$, and using it we may write the right
$G$-action as a map $\Phi\colon X\times G\to G\times X$, determining
the structure of $B$.

Important examples of bisets come from dynamics: let
$f\colon\mathcal M'\to \mathcal M$ be a partial self-covering of a
topological space $\mathcal M$, defined on a subset
$\mathcal M'\subseteq\mathcal M$. Fix a basepoint $*\in\mathcal M$,
and set $G\coloneqq\pi_1(\mathcal M,*)$. Set
\begin{equation}\label{eq:B(f)}
  B(f)\coloneqq\{\gamma\colon[0,1]\to\mathcal M\mid \gamma(0)=f(\gamma(1))=*\}\,/\,\text{homotopy},
\end{equation}
with left $G$-action given by pre-concatenation and right $G$-action
given by post-concatenation of the unique $f$-lift making the
resulting path continuous. A basis of $B$ consists of, for every
$z\in f^{-1}(*)$, of a choice of path in $\mathcal M$ from $*$ to
$z$.

Of particular interest, for us, is $f\colon(S^2,A)\selfmap$ a Thurston
map, with $\mathcal M\coloneqq S^2\setminus A$ and
$\mathcal M'\coloneqq S^2\setminus f^{-1}(A)$.  Recall that two
Thurston maps $f_0\colon(S^2,A_0)\selfmap$ and
$f_1\colon(S^2,A_1)\selfmap$ are called \emph{combinatorially
  equivalent} if there is a path
$(f_t\colon(S^2,A_t)\selfmap)_{t\in[0,1]}$ of Thurston maps joining
them. Kameyama proved in~\cite{kameyama:thurston}, in another
language, that $f_0,f_1$ are combinatorially equivalent if and only
$B(f_0)$ and $B(f_1)$ are \emph{conjugate}: setting
$G_i=\pi_1(S^2\setminus A_i,*_i)$ for $i=0,1$, there exists a
homeomorphism $\phi\colon S^2\setminus A_0\to S^2\setminus A_1$ and a
bijection $\beta\colon B(f_0)\to B(f_1)$ with
$g\cdot b\cdot h=\phi_*(g)\cdot \beta(b)\cdot\phi_*(h)$ for all
$g,h\in G_0,b\in B(f_0)$. See~\cite{bartholdi-dudko:bc2} for details.

Bisets may be composed: the product of two $G$-$G$-bisets $B,C$ is
$B\otimes_G C\coloneqq (B\times C)/\{(b g,c)=(b,g c)\}$, and is related
to the composition of partial coverings: we have a natural isomorphism
$B(g\circ f)\cong B(f)\otimes B(g)$. If $B,C$ are left-free with
respective bases $S,T$, then $B\otimes C$ is left-free with basis
$S\times T$.

\begin{defn}[\cite{nekrashevych:ssg}*{Definition~2.11.8}]\label{defn:contracting}
  Let $B$ be a $G$-$G$-biset. It is called \emph{contracting} if for
  some basis $X\subseteq B$ there exists a finite subset
  $N\subseteq G$ with the following property: for every $g\in G$ and
  every $n$ large enough we have the inclusion
  $X^{\otimes n}g\subseteq N X^{\otimes n}$ in $B^{\otimes n}$.
\end{defn}
Recall from~\cite{nekrashevych:ssg}*{Proposition~2.11.6} that, if $B$
is contracting for some basis $X$, then it is contracting for every
basis, possibly with a different $N$.  The set $N$ in
Definition~\ref{defn:contracting} is certainly not unique; but for
every basis $X$ there exists a minimal such $N$, written $N(B,X)$ and
called the \emph{nucleus} of $(B,X)$.

Recall also from~\cite{nekrashevych:ssg}*{Proposition~2.11.3} that, if
$G$ is finitely generated, $X$ is a basis of $B$ and $B$ is right
transitive, then $G$ is generated by $N(B,X)$. These hypotheses are
satisfied by bisets of Thurston maps,
see~\cite{bartholdi-dudko:bc2}*{Definition~\ref{bc2:dfn:SphBis}}. It
is then convenient to express the structure of $B$ by a \emph{Mealy
  automaton}: it is a finite directed labeled graph with vertex set
$N(B,X)$, with labels in $X\times X$ on edges, and with an edge labeled
`$x\to y$' from $g\in N(B,X)$ to $h\in N(B,X)$ whenever the equality
$x g=h y$ holds in $B$. In fact, one may consider the graph
$\mathfrak G$ with vertex set $G$ and an edge from $g\in G$ to
$h\in G$ labeled `$x\to y$' whenever $x g=h y$ holds in $B$, and then
$N(B,X)$ is precisely the forward attractor of $\mathfrak G$: an
equivalent formulation of Definition~\ref{defn:contracting} is that
every infinite path in $\mathfrak G$ eventually reaches $N(B,X)$ where
it stays. Here is an example of automaton, to which we shall return:
\begin{equation}\label{eq:basilica}
  \begin{fsa}[baseline]
    \node[state] (a) at (-2.7,1.3) {$a$};
    \node[state] (b) at (-2.7,-1.3) {$b$};
    \node[state] (e) at (0,0) {$1$};
    \path (a) edge node[above=2mm] {$0\to1$} (e) edge[bend left] node {$1\to0$} (b)
    (b) edge node[below=2mm] {$0\to0$} (e) edge[bend left] node {$1\to1$} (a)
    (e) edge[loop right] node[above=2mm] {$0\to0$} node[below=2mm] {$1\to1$} (e);
  \end{fsa}
\end{equation}
In this automaton, we have $X=\{0,1\}$ and $G=\langle a,b\rangle$. The
biset structure is determined by the equations
\[0\cdot a=1,\quad 1\cdot a=b\cdot0,\quad 0\cdot b=0,\quad 1\cdot
  b=a\cdot1.\] The reader may check that this is the biset $B(f)$ as
defined in~\eqref{eq:B(f)} for the partial self-covering $f(z)=z^2-1$
of $\hC\setminus\{0,-1,\infty\}$.

\begin{prop}
  Let $G$ be a finitely generated group with solvable word problem,
  and let $B$ be a computable left-free biset (namely, for a basis $X$
  the structure map $X\times G\to G\times X$ is computable). Then it
  is \emph{semi-decidable} whether $B$ is contracting: there is an
  algorithm that either runs forever (if $G$ is not contracting) or
  returns $N(B,X)$ (if $G$ is contracting).
\end{prop}
\begin{proof}
  Assume that $B$ is contracting, with nucleus $N(B,X)$. Denote by
  $\mathscr P_f(G)$ the set of finite subsets of $G$, and define the
  self-map $\phi\colon\mathscr P_f(G)\selfmap$ by
  \[\phi(A)=\{h\in G\mid\text{ there exist $x,y\in X$ with }h y \in x A\}.
  \]
  Clearly $\phi$ is computable, and $\phi(N(B,X))=N(B,X)$. For
  $A\subseteq G$ finite, set
  $\psi(A)\coloneqq\bigcup_{k\ge0}\phi^k(A)$. The sequence
  $(\bigcup_{k=0}^n\phi^k(A))_n$ is ascending and eventually all
  $\phi^k(A)$ are contained in $N(B,X)$, so
  $\psi\colon\mathscr P_f(G)\selfmap$ is computable. Again for
  $A\subseteq G$ finite, set
  $\omega(A)\coloneqq\bigcap_{n\ge0}\phi^n(\psi(A))$. The sequence
  $(\phi^n(\psi(A)))_n$ is a decreasing subsequence of the finite set
  $\psi(A)$, so $\omega$ is also computable.

  Let $S$ be a finite generating set for $G$, and assume
  $1\in S=S^{-1}$. Set $N_0\coloneqq \{1\}$, and for $n\ge1$ set
  $N_n\coloneqq\omega(N_{n-1}S)$. Then $(N_n)_n$ is an increasing
  subsequence of $N(B,X)$, so it stabilizes, and its limit
  $\bigcup_{n\ge0} N_n$ is computable and equals $N(B,X)$.
\end{proof}

\subsection{Limit spaces}\label{ss:limit}
Let $B$ be a contracting $G$-$G$-biset, and let $X$ be a basis of
$B$. Define a relation on $X^\infty$, called \emph{asymptotic
  equivalence}, by
\begin{multline*}
  (x_1x_2\dots)\sim(y_1y_2\dots)\Longleftrightarrow\\
  \exists(g_0,g_1,g_2,\dots)\in
  G^\infty\text{ with $\#\{g_n\}<\infty$ and }x_n g_n=g_{n-1}y_n\text{
    for all }n\ge1.
\end{multline*}
More precisely, one says in this case that $x_1x_2\dots$ and
$y_1y_2\dots$ are \emph{$g_0$-equivalent}.  The \emph{limit space} of
$B$ is the quotient
\[\Julia(B)\coloneqq X^\infty/{\sim}.\]
More precisely, it is a topological orbispace, with at class
$[x_1 x_2\dots]\in\Julia(B)$ the isotropy group
$\{g_0\in G\mid x_1x_2\dots\text{ is $g_0$-equivalent to itself}\}$.

By~\cite{nekrashevych:ssg}*{Theorem~3.6.3}, we have
$x_1x_2\dots\sim y_1y_2\dots$ if and only if there exists a
left-infinite path in the Mealy automaton of $B$, with labels
$\dots,x_2\to y_2,x_1\to y_1$ on its arrows. These sequences are
$g_0$-equivalent for $g_0$ the terminal vertex of the path. In
particular, $\sim$-equivalence classes have cardinality at most
$\#N(B,X)$.  The topological (orbi)space $\Julia(B)$ is compact,
metrizable, of finite topological dimension. For example,
\eqref{eq:basilica} gives
$x_1\dots x_n0(11)^\infty\sim x_1\dots x_n1(10)^\infty$ for all
$x_1,\dots,x_n\in X=\{0,1\}$.

Denote by $s\colon X^\infty\selfmap$ the shift map
$x_1x_2x_3\dots\mapsto x_2x_3\dots$. Clearly the asymptotic
equivalence is invariant under $s$, so $s$ induces a self-map
$s\colon\Julia(B)\selfmap$.
By~\cite{nekrashevych:ssg}*{Corollary~3.6.7}, the dynamical system
$(\Julia(B),s)$ is independent, up to topological conjugacy, of
the choice of $X$. Note that $s$ only induces a partial self-covering
of $\Julia(B)$, if the orbispace structure of $\Julia(B)$ is
taken into account.

Let $f\colon\mathcal M'\to\mathcal M$ be a partial self-covering as above, and
assume that $\mathcal M$ has a complete length metric which is
expanded by $f$. The \emph{Julia set} of $f$ is defined as the
accumulation set of backward iterates of a generic point: fix
$z\in\mathcal M$, and define
\begin{equation}\label{eq:julia}
  \Julia(f)\coloneqq\overline{\bigcap_{n\ge0}\bigcup_{m\ge n}f^{-m}(z)},
\end{equation}
a definition that does not depend on the choice of $z$.

By~\cite{nekrashevych:ssg}*{Theorem~5.5.3} the biset $B(f)$ defined
in~\eqref{eq:B(f)} is contracting, and the dynamical systems
$(\Julia(f),f)$ and $(\Julia(B(f)),s)$ are conjugate.

The following image shows the Julia set of $f(z)=z^2-1$, the loops
$a,b\in\pi_1(\hC\setminus\{0,-1,\infty\},*)$, and the basis
$\{\ell_{x_0},\ell_{x_1}\}$ that were used to compute the
automaton~\ref{eq:basilica} ($\ell_{x_1}$ is so short that it is not
visible):
\begin{center}
  \begin{tikzpicture}[>=stealth]
\useasboundingbox (-6,-2) rectangle (6,2);
\node[gray] at (0,0) {\includegraphics[width=12cm,height=4cm]{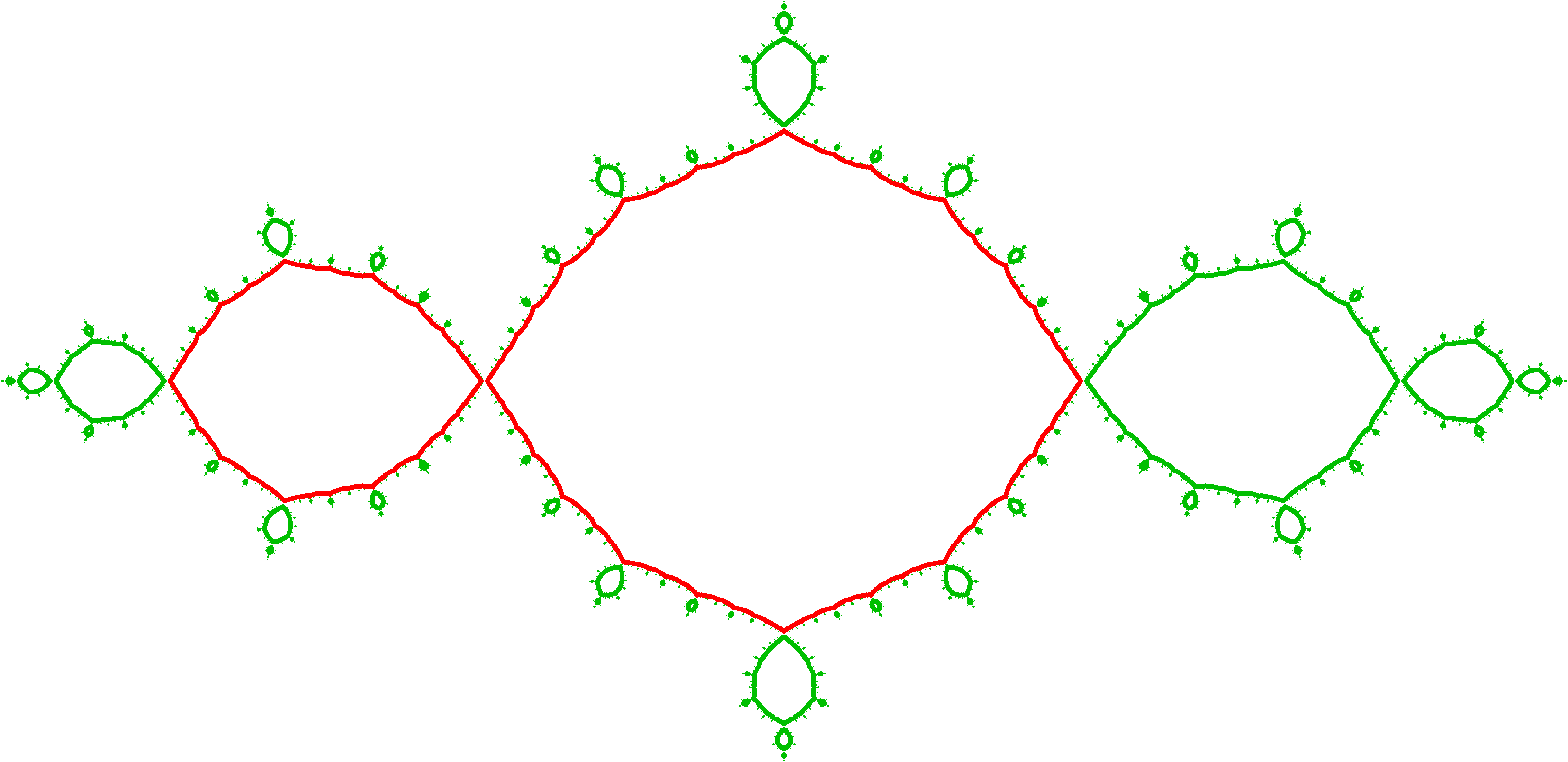}};
\node at (0,0) {$0$};
\node at (-3.6,0) {$-1$};
\node[inner sep=0pt] (*) at (-2.1,0) {$*$};
\node[inner sep=0pt] (x1) at (-2.6,0) {$x_1$};
\node[inner sep=0pt] (x0) at (2.6,0) {$x_0$};
\draw[thick,->] (*) .. controls (2.5,-4) and (2.5,4) .. node[left] {$a$} (*);
\draw[thick,->] (*) .. controls (-5,2.3) and (-5,-2.3) .. node [very near end,below] {$b$} (*);
\draw[thick,dashed,->] (x1) .. controls (-5.3,2) and (-5.3,-2) .. node [near start,above left=-1mm] {$f^{-1}(a)$} (x1);
\draw[thick,dashed,->] (x0) .. controls (5.3,-2) and (5.3,2) .. node [near start,below right=-1mm] {$f^{-1}(a)$} (x0);
\draw[thick,dashed,->] (x0) .. controls (1.4,2) and (-1.4,2) .. node [near start,above right=-1mm] {$f^{-1}(b)$} (x1);
\draw[thick,dashed,->] (x1) .. controls (-1.4,-2) and (1.4,-2) .. node [near start,below left=-1mm] {$f^{-1}(b)$} (x0);
\draw[thick,dotted,->] (*) -- (x1);
\draw[thick,dotted,->] (*) .. controls (-1.3,0.3) and  (1,1.6) .. node [below] {$\ell_{x_0}$} (x0);
\end{tikzpicture}
\end{center}

\subsection{Orbisphere contracting bisets}\label{ss:orbisphere}
We slightly modify the definition of ``contracting'' for sphere
bisets, because of the orbisphere structures. Let ${}_GB_G$ be a
sphere biset with $G=\pi_1(S^2,A)$. Recall
from~\cite{bartholdi-dudko:bc2}*{Equation~\eqref{bc2:eq:minorbispace}} that
there is a minimal orbisphere structure $\ord_B$ given by $B$. We call
an orbisphere structure $\ord\colon A\to\{2,3,\dots,\infty\}$
\emph{bounded} if $\ord(a)=\infty\Leftrightarrow\ord_B(a)=\infty$ and
$\ord(a)\deg_a(B)\mid\ord(B_*(a))$ for all $a\in A$. Let $\overline G$
denote the quotient orbisphere group
$G/\langle\gamma_a^{\ord(a)}:a\in A\rangle$. Then we call $B$ an
\emph{orbisphere contracting} biset if
$\overline G\otimes_G B\otimes_G\overline G$ is contracting for some
bounded orbisphere structure on $(S^2,A)$.

\section{Expanding non-torus maps}
Our purpose is, in this section, to endow the sphere $(S^2,A)$ with a
smooth metric that is expanded by a self-map
$f\colon (S^2,A)\selfmap$. We recall that by $A^\infty\subset A$ we
denote the forward orbit of periodic critical points of $f$.  A
\emph{non-torus} map is a map that is not doubly covered by a torus
endomorphism.

\begin{defn}[Metrically expanding maps]
  Consider a Thurston map $f\colon (S^2,A)\selfmap$ and let $A'$ be
  a forward-invariant subset of $A^\infty$. We say that $f$ is
  \emph{metrically expanding} if there exists a length metric $\mu$ on
  $S^2\setminus A^\infty$ such that
 \begin{enumerate}
 \item for every non-trivial rectifiable curve
   $\gamma\colon[0,1]\to S^2\setminus A'$ the length of any lift of
   $ \gamma$ under $f$ is strictly less than the length of $\gamma$;
   and \label{defn:MetrExp:ExpCond}
 \item at all $a\in A'$ the first return map of $f$ is locally
   conjugate to $z\mapsto z^{\deg_a(f ^n)}$.\label{defn:MetrExp:NormCond}
 \end{enumerate}
  
 If $A'=A^\infty$, then $f\colon (S^2,A)\selfmap$ is called a
 \emph{B\"ottcher expanding} map.
\end{defn}

If $\mu=\dd s$ is a Riemannian orbifold metric on $(S^2,A)$
(i.e.~$\mu$ is a smooth metric on $S^2\setminus A'$ with possible cone
singularities in $A\setminus A'$), then
Condition~\eqref{defn:MetrExp:ExpCond} may be replaced by
$f^*\dd s<\dd s$.

Let us now define a more general notion of topological
expansion. Consider first a covering map
$f\colon \mathcal M'\to \mathcal M$ between compact topological
orbispaces, with $\mathcal M'\subseteq \mathcal M$. We call $f$
\emph{topologically expanding} if there exists a finite open covering
$\mathcal M'=\bigcup\mathcal U_i$ such that connected components of
$f^{-n}(\mathcal U_i)$ get arbitrarily small as $n\to\infty$, in the
sense that for every finite open covering
$\mathcal M=\bigcup\mathcal V_j$ there exists $n\in\N$ such that every
connected component of every $f^{-n}(\mathcal U_i)$ is contained in
some $\mathcal V_j$. Equivalently, the diameter of connected
components of $f^{-n}(\mathcal U_i)$ tends to $0$ with respect to any
metric on $\mathcal M'$.

\begin{defn}[Topological expanding maps]\label{defn:TopExp}
  Consider a Thurston map $f\colon(S^2,A)\selfmap$ and let $A'$ be a
  forward-invariant subset of $A^\infty$. We call $f$
  \emph{topologically expanding} if there exist
  $\mathcal M'\subseteq \mathcal M\subseteq S^2$ compact with a
  topologically expanding orbifold covering map
  $f\colon (\mathcal M',A)\to (\mathcal M,A)$, such that every
  connected component $\mathcal U$ of $S^2\setminus \mathcal M$ is a
  disk containing a unique point $a\in A'$, and $\mathcal U$ is
  attracted to $a$, and the first return of $f$ is locally conjugate
  to $z\mapsto z^{\deg_a(f ^n)}$ at $a$.

  If $A'=A^\infty$, then $f\colon (S^2,A)\selfmap$ is called a
  \emph{B\"ottcher topologically expanding} map.
\end{defn}

\begin{prop}\label{prop:metr=>top}
  A metrically expanding map is topologically expanding.
\end{prop}
\begin{proof}
  Let $f\colon(S^2,A)\selfmap$ be metrically expanding. For each point
  $a\in A'$ choose an open neighborhood $\mathcal U_a\ni a$ such that
  $f(\mathcal U_a)$ is compactly contained in $\mathcal U_{f(a)}$. Set
  $\mathcal M=S^2\setminus\bigcup\mathcal U_a$ and
  $\mathcal M'=f^{-1}(\mathcal M)$.
\end{proof}

\noindent The goal of this section is to prove the following criterion:
\begin{thm}[Expansion Criterion]\label{thm:ExpCr}
  The following are equivalent, for a combinatorial equivalence class
  $\mathscr F$ of Thurston maps:
  \begin{enumerate}
  \item $\mathscr F$ contains a metrically B\"ottcher expanding map;\label{thm:ExpCr:1}
  \item $\mathscr F$ contains a topologically expanding map;\label{thm:ExpCr:3}
  \item $B(f)$ is an orbisphere contracting biset for every
    $f\in\mathscr F$;\label{thm:ExpCr:4}
  \item $\mathscr F$ does not admit a Levy cycle, and if $\mathscr F$
    is doubly covered by a torus endomorphism
    $M z+v \colon \R/\Z \selfmap$ then both eigenvalues of $M$ have
    absolute value greater than $1$. \label{thm:ExpCr:5}
  \end{enumerate}

  Furthermore, if any of these properties hold then the expanded
  metric may be assumed to Riemannian of pinched negative curvature.
\end{thm}

We will prove Theorem~\ref{thm:ExpCr} for maps not doubly covered by
torus endomorphisms. The remaining case follows from
\cite{haissinsky-pilgrim:algebraic}*{Theorem~4} or from
\cite{selinger-yampolsky:geometrization}.  The hardest implication in
the proof is \eqref{thm:ExpCr:5}$\Rightarrow$\eqref{thm:ExpCr:1}, and
will occupy most of this section.
\begin{proof}[Proof of Theorem~\ref{thm:ExpCr}, \eqref{thm:ExpCr:1}$\Rightarrow$\eqref{thm:ExpCr:3}$\Rightarrow$\eqref{thm:ExpCr:4}$\Rightarrow$\eqref{thm:ExpCr:5}]\label{sec:ProofthmExpCr}
  The implication \eqref{thm:ExpCr:1}$\Rightarrow$\eqref{thm:ExpCr:3}
  follows from Proposition~\ref{prop:metr=>top}.
  By~\cite{bartholdi-h-n:aglt}*{Proposition~6.4}, the biset of a
  topologically expanding map is contracting; this is  \eqref{thm:ExpCr:3}$\Rightarrow$\eqref{thm:ExpCr:4} with slight adjustments to sphere maps.

  Consider next a combinatorial equivalence class $\mathscr F=[f]$
  admitting a Levy cycle
  $(\gamma_0,\gamma_1,\dots,\gamma_n=\gamma_0)$. Write
  $G=\pi_1(S^2\setminus A,*)$, consider the $G$-$G$-biset $B(f)$, and
  choose a basis $X$ for it. The assumption that $(\gamma_i)_i$ is a
  Levy cycle means that there exist basis elements
  $x_0,x_1,\dots,x_n=x_0\in X$ with $x_i\gamma_{i+1}=\gamma'_i x_i$
  and $\gamma'_i$ conjugate to $\gamma_i$ for all $i\in\Z/n$. In
  particular, for every $j\in\Z$ there is a conjugate of $\gamma_0^j$
  in the nucleus of $(B(f),X)$. Now $\gamma_0$ has infinite order in
  $G$, because it is not peripheral. It follows that the nucleus of
  $(B(f),X)$ is infinite, so $B(f)$ is not orbisphere contracting.
\end{proof}

\begin{proof}[Outline of the proof of Theorem~\ref{thm:ExpCr}, \eqref{thm:ExpCr:5}$\Rightarrow$\eqref{thm:ExpCr:1}]
  We wish to prove that a Levy-free non-torus Thurston map $f$ admits
  an expanding metric. We do so by explicitly constructing the metric
  adapted to $f$.

  We consider the \emph{decomposition} of $S^2$ into \emph{small
    spheres} along the canonical obstruction $\CC_f$. The map $f$
  restricts to maps between the small spheres, well-defined up to
  isotopy; and the \emph{small Thurston maps} --- the return maps to
  small spheres --- are combinatorially equivalent to rational maps.

  We first isotope the periodic small spheres so into complex spheres,
  in such a manner that the small Thurston maps are rational. We put
  the hyperbolic metric on these periodic small spheres, and pull it
  back to preperiodic small spheres.

  It remains to attach the small spheres together. They are spheres
  with cusps; some of the cusps correspond to the marked set $A$, and
  some to $\CC_f$. Cut the cusps corresponding to $\CC_f$ along a very
  small horocycle, and connect the small spheres by very long and thin
  cylinders along the combinatorics of the original decomposition. We
  have constructed a space $X$ with a piecewise-smooth non-positively
  curved metric.

  Define a self-map $F\colon X\selfmap$ as follows: away from the
  truncated cusps, apply the original map $f$. Subdivide the long
  cylinders into long ``annuli'' and short ``annular spheres''. Map
  the annular spheres to the small spheres they originally mapped to,
  and map the annuli affinely to each other.

  The map $F$ is expanding: on periodic small spheres, because it is
  modelled on rational maps; on preperiodic small spheres, too; on
  annular and trivial small spheres, because they are contained in
  thin cylinders; and on annuli because of properties of the canonical
  obstruction: it contains neither Levy cycles nor primitive unicycles.
\end{proof}

\subsection{Conformal metrics}
Recall first that every Riemannian metric $s$ on a surface (for
example a sphere) admits local isothermal coordinates; i.e.\ there is
a local chart $\mathcal U$ where $\dd s$ takes form $\rho(z)|\dd z|$
on the tangent space of $\mathcal U$; the function
$\rho\colon \mathcal U\to\R_+$ should be smooth. A metric in this form
is called \emph{conformal}. The Gaussian curvature
$\kappa\colon \mathcal U \to\R$ is given by
\[\kappa(z)=\frac{-\Delta(\log\rho(z))}{\rho(z)^2},
\]
by an easy calculation (see
e.g.~\cite{griffiths-harris:pag}*{page~77}). We
note for future reference the following simple calculation: if
$\rho(z)=\sigma(|z|)$ is rotationally invariant around $0\in  \mathcal U$
in the chart $z$, then
the Gaussian curvature may be computed as
\begin{equation}\label{eq:polarkappa}
  \kappa(z)=-\frac{\log(\sigma)''+\log(\sigma)'/|z|}{\sigma(|z|)^2}.
\end{equation}

We shall consider conformal metrics $s$ on an orbisphere
$(S^2,A)$. This means that in a suitable coordinates we have
2$\dd s=\rho(z)|\dd z|$ with $\rho\colon S^2\setminus A\to \R_{+}$
that has continuous extension
$\rho\colon S^2\to \R_{+}\cup \{+\infty\}$ such that if for $a\in A$

\begin{itemize}
\item $\rho(a)<+\infty$, then $\rho$ is smooth at $a$ (i.e.~$a$ is a
  usual point);
\item $\rho(a)=+\infty$ but $a$ at finite distance from points in
  $S^2$, then $(S^2,s)$ around $a$ is a quotient of a chart
  $\mathcal U$ endowed with a conformal metric under a finite group of
  isometries; the point $a$ is called a \emph{cone singularity}.
\end{itemize}
If $\rho(a)=+\infty$ and $a$ at infinite distance from points in
$S^2$, then $a$ is called a \emph{cusp}.

\subsection{Fatou and Julia sets}\label{ss:fatou}
We adapt the definition of Julia sets from~\eqref{eq:julia} to
expanding Thurston maps. We recall some well-known facts, and include
their proofs for convenience.
\begin{defn}
  Let $f\colon(S^2,A)\selfmap$ be an expanding Thurston map. Its \emph{Julia
    set} $\Julia(f)$ is the closure of the set of repelling
  periodic points, namely the closure of the set of points $z\in S^2$
  with $f^n(z)=z$ for some $n>0$ but admitting no neighbourhood
  $\mathcal U\ni z$ with $f^n(\mathcal U)$ compactly contained in
  $\mathcal U$.

  The \emph{Fatou set} $\Fatou(f)$ is the locus of continuity of
  forward orbits, namely the set of $z\in S^2$ at which the orbit map
  $S^2\to (S^2)^\infty,z\mapsto(z,f(z),f^2(z),\dots)$ is continuous in
  supremum norm (of any metric on $S^2$ realizing its topology).
\end{defn}

\begin{lem}\label{lem:JulFatDecomp}
  $S^2=\Julia(f)\sqcup\Fatou(f)$. Moreover, in the notation of
  Definition~\ref{defn:TopExp} the Julia set $\Julia(f)$ is the set of
  points in $\mathcal M'$ that do not escape $\mathcal M'$ under iteration of $f$.
\end{lem}
\begin{proof}
  By definition, every point $z$ escaping $\mathcal M'$ is in the attracting
  basin of $A'$, so $z$ has a stable orbit and $z\in
  \Fatou(f)$. Conversely, suppose that $z$ does not escape $\mathcal M'$. Fix a
  metric on $S^2$ realizing its topology and consider $\varepsilon>0$
  such that for every $\mathcal V\subset \mathcal M$ with diameter less than
  $\varepsilon$ the components of $f^{-n}(\mathcal V)$ get arbitrarily
  small as $n\to\infty$. Choose a large $n\in \N$ and consider the
  $\varepsilon$-neighborhood $\mathcal V\subset \mathcal M$ of $f^{n}(z)$. The
  pullback of $\mathcal V$ along the orbit of $z$ is a small (since
  $n$ is large) neighborhood $\mathcal V'$ of $z$; so there are points
  close to $z$ that have orbits $\varepsilon$-away from the orbit of
  $z$.  This shows that $z\not\in\mathcal F(f)$.

  Choose now a small closed topological disc $\mathcal V$ containing
  $z$. There is an $n\ge 1$ such that
  $f^n(\mathcal V) \supset \mathcal V$. Therefore, there is a periodic
  point in $\mathcal V$. This shows that $z\in\Julia(f)$.
\end{proof}

The Fatou set of $f$ is open. Every periodic component of
$\mathcal F(f)$ contains a attracting periodic point called the
\emph{center}; this point belongs to $A'$. By
Lemma~\ref{lem:JulFatDecomp} every non-periodic component of
$\mathcal F(f)$ is preperiodic because it consists of points escaping
to $S^2\setminus \mathcal M$. We may now deduce that every component of
$\mathcal F(f)$ is an open topological disc.

Consider first a periodic connected component $O$ of $\Fatou(f)$ and
let $a\in A'\cap O$ be its center. There is a conjugacy from $O$ to
the open disk $\mathbb D\subset\C$ such that the first return map
$f^n\colon O\to O$ is conjugate to the map $z^{\deg_a(f^n)}$. We write
$\deg_O(f)\coloneqq\deg_a(f)$, and call the conjugacy
$\phi_O\colon O\to\mathbb D$ a \emph{B\"ottcher coordinate}. We may
then determine coordinates on every Fatou component on the forward and
backward orbit of $O$ in such a manner that, for every Fatou component
$U$, the restriction $f\restrict U\colon U\to f(U)$ is conjugate to a
monomial map by
$\phi_{f(U)}\circ f\restrict U=z^{\deg_U(f)}\circ\phi_U$.

We use B\"ottcher coordinates to define, in every Fatou component $O$,
\emph{internal rays} $R_{O,\theta}\subset O$ by
\[R_{O,\theta}=\phi_O^{-1}\{r e^{2i\pi\theta}\mid r<1\}.
\]
These rays are mapped to each other by
$f(R_{O,\theta})=R_{f(O),\deg_O(f)\theta}$.  The following statement
follows immediately from the existence of B\"ottcher coordinates:
\begin{lem}\label{lem:Bottcher extension}
  Let $f\colon(S^2,A)\selfmap$ be a B\"ottcher map, and let $a\in A$
  be a degree-$d$ attracting point. Let $F$ denote its immediate basin
  of attraction; then $F$ is a connected component of the Fatou set of
  $f$. Let $\overline D$ denote the compactification of
  $S^2\setminus\{a\}$ by adding a circle of directions in replacement
  of $a$; then $f$ extends continuously to a self-map of
  $\overline D$, such that the boundary circle is mapped to itself by
  $z\mapsto z^d$.\qed
\end{lem}

\subsection{Canonical obstructions and decompositions}
We shall make essential use of Pilgrim's \emph{canonical
  decomposition}. Let $f\colon (S^2,A)\selfmap$ be a Thurston
map. Then there is an induced pullback map $f^*$ on the Teichm\"uller
space $\mathscr T_A$ of complex structures on $(S^2,A)$,
see~\cite{bartholdi-dudko:bc2}*{\S\ref{bc2:ss:examples}}; for a given
complex structure $\eta$, the pullback $f^*\eta$ is defined such that
the map $f\colon (S^2,A,f^*\eta)\to(S^2,A,\eta)$ is holomorphic. The
map $f$ is combinatorially equivalent to a rational map if and only if
$f^*$ has a fixed point.

Let $\gamma$ be an essential simple closed curve and let
$\eta\in \mathscr T_A$ be a complex structure. The length
$\langle\gamma , \eta\rangle$ of $\gamma$ with respect to $\eta$ is
defined as the length of the unique geodesic in $(S^2,A,\eta)$ that is
homotopic to $\gamma$. This defines an analytic function
$\langle\gamma,{-}\rangle\colon \mathscr T_A\to \R$. 

\begin{defn}[Canonical obstruction~\cite{pilgrim:combinations}*{Theorem~1.2}]
  Let $f\colon (S^2,A)\selfmap$ be a Thurston map, and consider
  $\eta \in T_A$.

  The \emph{canonical obstruction} $\CC_f$ is the set of homotopy
  classes of essential simple closed curves $\gamma$ such that
  $\langle\gamma,f^{n*}\eta\rangle$ tends to $0$ as $n$ tends to
  infinity.
\end{defn}
It follows from the following theorem that the definition of $\CC_f$
does not depends on $\eta$. It was proved by Kevin Pilgrim that
$\CC_f$ is a multicurve.

\begin{thm}[Pilgrim, \cite{pilgrim:cto}]\label{th:kevin_can_obstr}
  If $\CC_f$ is empty and the degree of $f$ is at least $2$, then $f$
  is combinatorially equivalent to a rational map.\qed
\end{thm}

For $f$ a Thurston map, its \emph{canonical decomposition} is the
collection of spheres and annuli obtained by cutting $f$ along the
canonical obstruction $\CC_f$. Recall that the \emph{small Thurston
  maps} are the return maps of $f$ to the small spheres in a
decomposition.

\begin{thm}[Pilgrim, Selinger~\cite{selinger:augts}]\label{thm:CanDecomp}
  Every small Thurston map in the canonical decomposition of $f$ is
  either
  \begin{itemize}
  \item combinatorially equivalent to a rational non-Lattes
    post-critically finite map;
  \item double covered by a torus endomorphism;
  \item \phantombullet{or}a homeomorphism.
  \end{itemize}
\end{thm}
Theorem~\ref{thm:CanDecomp} was conjectured by Kevin Pilgrim (who also
proved a slightly weaker version of this theorem,
see~\cite{pilgrim:combinations}*{page~13}) and was eventually proved
by Nikita Selinger.

\subsection{Construction of the model}\label{subsec:ProofOfThmExpCrMain}
We give here the proof the
implication~\eqref{thm:ExpCr:5}$\Rightarrow$\eqref{thm:ExpCr:1}, by
constructing a negatively curved Riemannian metric on $X\simeq S^2$
and an expanding map $F\colon X\selfmap$ isotopic to $f$; see
Figure~\ref{Fig:ThExpCr} for an illustration of the
construction.

\begin{figure}
  \begin{tikzpicture}
    \def\leftsphere#1#2{+(0,0.1) .. controls +(180:#1/2) and +(0:#1/2) .. +(-#1,#2)
      .. controls +(180:#2) and +(180:#2) .. +(-#1,-#2)
      .. controls +(0:#1/2) and +(180:#1/2) .. +(0,-0.1) ++(0,0)}
    \def\rightsphere#1#2{+(0,0.1) .. controls +(0:#1/2) and +(180:#1/2) .. +(#1,#2)
      .. controls +(0:#2) and +(0:#2) .. +(#1,-#2)
      .. controls +(180:#1/2) and +(0:#1/2) .. +(0,-0.1) ++(0,0)}
    
    \fill[gray!30] (-1.5,0) \leftsphere{2.5}{0.8} (1.5,0) \rightsphere{2.5}{0.8};
    \fill[white] (-1.4,0) circle (2mm) (1.4,0) circle (2mm);
    \draw[very thick] (-1.5,0) node [xshift=-22mm] {$S_1$}
    \leftsphere{2.5}{0.8} +(0,0.1) -- +(3,0.1) +(0,-0.1) -- node [below] {$T$} +(3,-0.1) (1.5,0) node[xshift=22mm] {$S_2$}
    \rightsphere{2.5}{0.8};

    \fill[gray!30] (-1.8,2.5) \leftsphere{2.2}{0.8} ++(3.6,0)
    \rightsphere{2.2}{0.8}; \fill[gray!30] (-0.8,2.5) +(0,-0.1)
    .. controls +(0:0.4) and +(180:0.4) .. +(0.8,-0.5) .. controls
    +(0:0.4) and +(180:0.4) .. +(1.6,-0.1) -- +(1.6,0.1) .. controls
    +(180:0.4) and +(0:0.4) .. +(0.9,0.5) -- +(0.9,0.6) -- +(0.7,0.6)
    -- +(0.7,0.5) .. controls +(180:0.4) and +(0:0.4) .. +(0,0.1)
    {[rotate around={-90:+(0.7,-0.1)}] \leftsphere{0.5}{0.3}};
    \fill[white] (-1.6,2.5) circle (4mm) (-1.1,2.5) circle (4mm)
    (1.1,2.5) circle (4mm) (1.6,2.5) circle (4mm) (0,3.1) circle (1.2mm);

    \draw[very thick] (-1.8,2.5) node [xshift=-19mm] {$S'_1$}
    \leftsphere{2.2}{0.8} +(0,0.1) -- +(1,0.1) +(0,-0.1) -- +(1,-0.1) ++(1,0) node [xshift=8mm] {\small $S'_2$}
    +(0,0.1) .. controls +(0:0.4) and +(180:0.4) .. +(0.7,0.5) -- +(0.7,0.6)
    {node [xshift=1mm,yshift=4mm] {\tiny $S'''_1$} [rotate around={-90:+(0.7,-0.1)}] \leftsphere{0.5}{0.3}}
    +(0.9,0.6) -- +(0.9,0.5) .. controls +(0:0.4) and +(180:0.4) .. +(1.6,0.1)
    +(0,-0.1) .. controls +(0:0.4) and +(180:0.4) .. +(0.8,-0.5)
    .. controls +(0:0.4) and +(180:0.4) .. +(1.6,-0.1) ++(1.6,0)
    +(0,0.1) -- +(1,0.1) +(0,-0.1) -- +(1,-0.1) ++(1,0)
    node [xshift=19mm] {$S''_1$} \rightsphere{2.2}{0.8};

    \draw (-1.3,2.4) -- node[left] {$2:1$} (-0.4,0.1) \fwdarrowonline{0.5};
    \draw (1.3,2.4) -- node[right] {$2:1$} (0.4,0.1) \bckarrowonline{0.5};
  \end{tikzpicture}
  \caption{Illustration to the Proof of Theorem~\ref{thm:ExpCr}. The
    map $f$ is indicated by the arrows, and sends $S'_1,S''_1,S'''_1$
    to $S_1$ and $S'_2$ to $S_2$. We first define a metric on the
    periodic small spheres ($S_1$), then on the preperiodic small
    spheres ($S_2$), and finally on the annuli between them. This map
    could be Pilgrim's ``blow-up an arc'' map,
    see~\cite{bartholdi-dudko:bc0}*{\S\ref{bc2:ss:pilgrim}}.}
  \label{Fig:ThExpCr}
\end{figure}
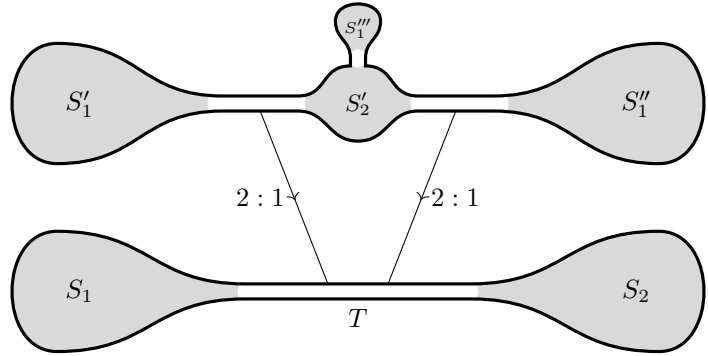

\subsubsection{Setup}
\label{sss:Prf:setup} 
The space $X$ is constructed by \emph{plumbing} between cusped
spheres: we enlarge the cusps to make them almost cylindrical, and
then truncate them and glue them on their common boundary. Three
variables dictate the construction: first a parameter $w\llcurly 1$ is
chosen; the perimeters of the ``cylindrical parts'' will lie between
$\pi w$ and $2\pi w$. Then a parameter $\ell\ggcurly 1/w$ is chosen;
the cylindrical parts will all have length between $\ell$ and
$2\ell$. Finally, a parameter $\epsilon\llcurly 1/\ell$ is chosen; it
will be a final adjustment to the construction that makes the
curvature bounded by $-\epsilon^2$ from above.

The map $F$ is very close to a rational map on each small sphere and
is very close to an affine map on each cylinder connecting small
spheres. After the main part of construction is carried we obtain a
metric $\mu$ that is weakly expanded by $F$ (namely, $F$ does not
contract $\mu$) and a certain iteration of $F$ is expands $\mu$. In
Lemma~\ref{lem:PertMetric} we perturb $\mu$ infinitesimally to make
$F$ expanding.

\subsubsection{The canonical decomposition}
Throughout this section, we let $\CC=\CC_f$ denote the canonical
obstruction of the Thurston map $f\colon(S^2,A)\selfmap$, and we
denote by $\Sph$ the collection of small spheres (components of
$S^2\setminus \CC$) of the canonical decomposition; so that
\[S^2=\bigsqcup_{\gamma\in\CC}\gamma\cup\bigsqcup_{S\in\Sph}S.
\]
As in~\cite{bartholdi-dudko:bc2}, for $S\in \Sph$ we denote by
$\widehat S$ the corresponding topological sphere marked by the image
of $A\cap S^2$ and the boundary curves. The map $f$ induces a map
$f\colon\Sph\selfmap$, and for each $S\in\Sph$ a map
$f\colon \widehat S\to \widehat{f(S)}$, well-defined up to isotopy,
see~\cite{bartholdi-dudko:bc2}*{Lemma~\ref{bc2:lem:SmallMaps}}.
 
Recall that we assumed that $f$ is a non-torus map: a map that is not
finitely covered by a torus endomorphism.
\begin{lem}\label{lem:ExpThmFirstObserv}
  If $f\colon(S^2,A)\selfmap$ is a Levy-free non-torus Thurston map
  and $\CC_f$ is non-empty, then $\CC_f$ is an anti-Levy Cantor
  multicurve, and all small Thurston maps in the canonical
  decomposition of $f$ are equivalent to non-torus rational maps.
\end{lem}
\begin{proof}
  Let us show that $\CC_f$ does not contain a primitive
  unicycle. Since a non-Levy unicycle has spectral radius strictly
  less than $1$, such a (primitive) cycle may not belong to $\CC_f$.

  Further, all small Thurston maps in $R(f,A,\CC)$ are non-torus and
  non-homeomorphisms, because torus and homeomorphism cycles can only
  be attached only via Levy cycles, because homeomorphisms and torus
  maps have no attracting periodic points. Theorem~\ref{thm:CanDecomp}
  concludes the proof.
\end{proof}

\subsubsection{Metrics on small spheres}
Consider a cycle of periodic spheres $S\to f(S)\to\dots\to
f^p(S)=S$. Let us denote by $\widehat S_i=\widehat{f^i(S)}$ the
topological sphere associated with $f^i(S)$ and we denote $A_i$ the
marked set of $S_i$. By Lemma~\ref{lem:ExpThmFirstObserv} the first
return map $f^p\colon \widehat S_i\selfmap$ is isotopic rel $A_i$ to a
rational map. Therefore, let us now assume that each $\widehat S_i$ is
a marked complex sphere and each $f_i\colon \widehat S_i\to S_{i+1}$
is a rational map. Choose next an orbifold structure
$\ord_i\colon A_i\to \{1,2,\dots, \infty\}$ such that
$f^p\colon (\widehat S_1,\ord_1)\selfmap$ is a partial self-covering
but is not a partial covering. We also choose $\ord_i$ in such a way
that $\ord_i(x)=\infty$ if and only if $x$ is in a periodic critical
cycle or $x$ is the image of a boundary curve.

We endow each $(\widehat S_i,\ord_i)$ with its natural hyperbolic
metric. Then every $f\colon \widehat S_i\to \widehat S_{i+1}$ is
either expanding (if it is not a covering) or an isometry (if it is a
covering); and $f^p\colon \widehat S_1\selfmap$ is expanding.
 
Similarly, we endow each preperiodic sphere $\widehat S'$, say marked
by $A'$, with a hyperbolic metric such that
$f\colon \widehat S'\to \widehat {f(S')}$ is either isometry or an
expanding map. The orbisphere structure
$\ord'\colon A'\to \{1,2,3,\dots, \infty\}$ is chosen so that
$\ord_i(x)=\infty$ if and only if $x$ is the image of a boundary
curve.
 
\subsubsection{Slight adjustment at cusps}\label{sss:AdjAtCusps}
For a periodic cycle $f\colon\widehat S_i\to \widehat S_{i+1}$ as
above consider a point $x\in \widehat S_i$ that is a cusp with respect
to the hyperbolic metric. Then a small neighbourhood of $x$ is
foliated by \emph{horocycles} --- curves perpendicular to geodesics
starting at $x$. We shall adjust locally the dynamics at cusps and
rescale there the hyperbolic metric by a factor of $1+\delta$ for
small $\delta$, in such a way that horocycles form an invariant
foliation of the new dynamics.

Suppose that $x\in S_1$ is periodic, say with period $q$.  Let
$\mathbb D^*\coloneqq \mathbb D\setminus \{0\}$ be the unit disc
punctured at $0$. Since $x$ is a cusp, the universal cover
$\mathbb D\to (S_1,\ord_1)$ factors as
$\mathbb D\to \mathbb D^*\overset{\pi_x}{\longrightarrow}
(S_1,\ord_1)$ with $\pi_x$ extended to $0$ by $\pi_x(0)=x$. Denote by
$H^*_r \subset \mathbb D^*$ the circle, i.e.~horocycle, centered at
$0$ with Euclidean radius $r$. For a sufficiently small $r$ the image
$H_r\coloneqq \pi_x(H^*_r)$ is a small simple closed curve around
$x\in \widehat S_1$.

Let $d>1$ be the local degree of $f^q$ at $x$ and let $U\subset S_1$
be the Fatou component containing $x$. Choose a B\"ottcher function
$B\colon U\to \mathbb D$ conjugating $f^p\colon U\selfmap$ to
$z\to z^d\colon \mathbb D\selfmap$. Denote by $E'_r \subset \mathbb D$
the circle centered at $0$ with Euclidean radius $r$. Then
$E_r\coloneqq B^{-1}(E')$ is an \emph{equipotential} of $U$. By
construction, $f^q(E_r)=E_{r/d}$.

Since $\pi_x$ and $B$ are conformal at $0$ there is a $\tau>0$ such
that $H_r$ approximates $E_{\tau r}$: for a sufficiently small $r$ the
horocycle $H_r$ lies in the $O(r^2)$-neighbourhood of $E_{\tau r}$ and,
moreover, the hyperbolic length of $E_{\tau r}$ is $-1/\log(r)+O(-r/\log r)$. (We
recall that $-1/\log(r)$ is the hyperbolic length of $H_r$.)

Choose now a small constant $\delta>0$ and a smooth function
$t\colon \R_{>0} \to \R_{>0}$ with $t(r)=1$ for $r\ge1$ and
$t(r)=1+\delta$ for $t<1/d$, such that the rescaled metric
$-t(r)/(r\log r) $ still has a negative curvature on $\mathbb D$. It
follows from~\eqref{eq:polarkappa} that $-t(r/R)/(r\log r)$ also has
negative curvature for all $R>1$. For a sufficiently large $R$ we
replace the hyperbolic metric around $x$ by $-t(r/R)/(r\log r)$ and we
adjust the dynamics of $f^q$ around $x$ such that $f^q$ maps $H_{r}$
to $H_{r/d}$ for all $r\le 1/R$. The adjustment is possible because
$f^q$ has expansion bounded away from $1$ around
$H_{1/R}\approx E_{\tau/R}$ with respect to the rescaled metric.

We now spread the adjusted dynamics along the orbit of $x$ as well as
to all preperiodic preimages of $x$ that are cusps with respect to
the hyperbolic metric. We perform the same operation at all cusps.

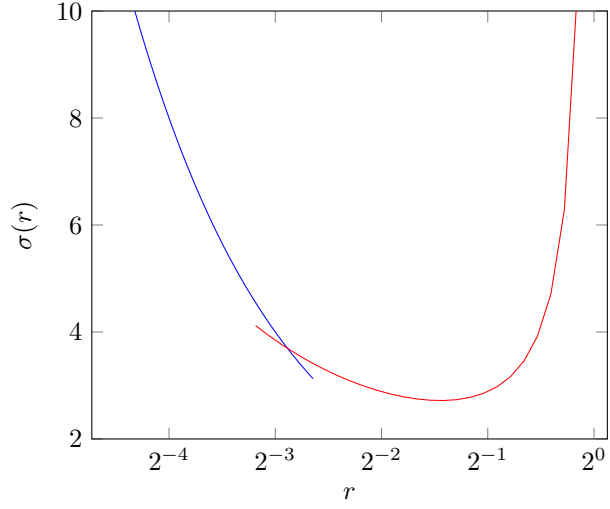
\begin{figure}
  \centering\begin{tikzpicture}
    \begin{semilogxaxis}[xlabel=$r$,ylabel=$\sigma(r)$,ymin=2,ymax=10,log basis x={2}]
      \addplot[domain=0.05:0.16,color=blue] {1/(2*x*cos(0.1*(ln(x))))};
      \addplot[domain=0.11:0.9,color=red] {-1/(x*ln(x))};
    \end{semilogxaxis}
  \end{tikzpicture}
  \caption{The curvature on the widened cusps}\label{fig:sigma}
\end{figure}

\subsubsection{Plumbing}\label{sss:Plumbing}
Let $S_1$ and $S_2$ be two hyperbolic small spheres with respective
cusps at $x_1\in S_1$ and $x_2\in S_2$. We now describe an operation,
\emph{plumbing}, that truncates $S_1$ and $S_2$ at $x_1$ and $x_2$
along their horocycles of perimeter $\approx 2\pi w$ and joins
$S_1,S_2$ along an almost flat cylinder with length $\approx \ell$
such that the resulting sphere still has a negatively curved
metric. Since $S_1$ and $S_2$ are covered by punctured discs, it is
sufficient to describe the operation between two copies
$\mathbb D^*_1, \mathbb D^*_2$ of the unit disc punctured at $0$. 

The hyperbolic metric on the unit disc punctured at $0$ is written as
$\sigma(|z|)|\dd z|$ with $\sigma(r)=-1/(r\log r)$. Replace $\{0<|z|<1\}$
by $\{\exp(-w\ell/2)\le|z|<1\}$, and give it a metric
$\sigma(|z|)|\dd z|$ with
\[\sigma(r)\approx\max\Big\{\frac1{w r\cos(\epsilon(\log(r)-w\ell/2))},\frac{-1}{r\log r}\Big\};
\]
see Figure~\ref{fig:sigma}. On that figure, the blue part
$1/(w r\cos(\epsilon(\log(r)-w\ell)))$ is a piece of the one-sheeted
hyperboloid of curvature $-\epsilon^2$, as can be readily checked
using~\eqref{eq:polarkappa}, with its unique minimal closed curve of
length $2\pi w$ appearing at radius $r=\exp(-w\ell/2)$, and with
length $\approx\ell/2$. The red part $-1/(r\log r)$ is the original
metric on the cusp. At $\approx-w\ell/2$ we replace $\sigma$ by a
smooth function that is slightly bigger than $\sigma (-w\ell/2)$; we
can do it such that $\log(\sigma)''\ggcurly 1$ at $\approx -w\ell/2$;
thus we guarantee that the new function still has a negative curvature
by~\eqref{eq:polarkappa}.

After the metrics on both cusps have been modified in the above
manner, they can be attached along their common boundary curve
$\{|z|=\exp(-w\ell/2)\}$, which is geodesic (it corresponds to the
core curve of the hyperboloid). The result is a space consisting of
two truncated discs with curvature $-1$ attached by a cylinder of
curvature $-\epsilon^2$, perimeter $\approx2\pi w$ and length
$\approx\ell$.

\subsubsection{Global metric}\label{sss:GlobMetr}
We now perform the plumbing between the metrized small spheres in
$\Sph$. The following proposition will allow us to endow the annuli of
the canonical decomposition with an expanding map.
\begin{prop}\label{prop:CombGl}
  There is an assignment
  \[\CC\to(1,2)\times(1,2),\qquad \gamma\mapsto(w_\gamma,\ell_\gamma)
  \]
  (where $w_\gamma$ is the ``width'' of the annulus corresponding to
  $\gamma$ and $\ell_\gamma$ is its ``length'') such that
  \begin{itemize}
  \item if for a non-peripheral curve $\delta\in f^{-1}(\CC)$ the map
    $f\colon \delta\to f(\delta)$ is one-to-one, then
    $w_{f(\delta)}>w_\delta$;
  \item if for a curve $\gamma\in\CC$ there is a unique non-peripheral
    curve $\delta\in f^{-1}(\CC)$ isotopic to $\gamma$, then
    $\ell_{f(\delta)}>\ell_\gamma$.
  \end{itemize}
\end{prop}
\begin{proof}
  We first note that only an ordering of the $(w_\gamma)$ and
  $(\ell_\gamma)$ is required; once such an ordering is found, they
  can easily be embedded in the interval $(1,2)$.

  If an assignment $\gamma\to w_\gamma$ is forbidden, then there is a
  sequence $\gamma_0,\gamma_1,\dots,\gamma_n=\gamma_0$ of curves in
  $\CC$ such that $w_{\gamma_{i+1}}>w_{\gamma_i}$ holds. This means
  that $\bigcup_i\partial \gamma_i$ contains a Levy cycle. This
  contradicts the assumption that $\CC_f$ is an anti-Levy multicurve,
  by Lemma~\ref{lem:ExpThmFirstObserv}.

  If an assignment $\gamma\to\ell_\gamma$ is forbidden, then there is
  a sequence $\gamma_0,\gamma_1,\dots,\gamma_n=\gamma_0$ of curves in
  $\CC$ such that $\ell_{\gamma_{i+1}}>\ell_{\gamma_i}$ holds. This
  means that $\bigcup_i\partial \gamma_i$ is a primitive unicycle, and
  this contradicts the assumption that $\CC_f$ is a Cantor multicurve,
  again by Lemma~\ref{lem:ExpThmFirstObserv}.
\end{proof}

We scale the solutions $(\ell_\gamma),(w_\gamma)$ given by
Proposition~\ref{prop:CombGl} so that $\ell \le \ell_\gamma \le 2\ell$
and $w/2\le w_\gamma\le w$, for the parameters
$\ell\ggcurly 1/w\ggcurly 1$ of the construction in~\S\ref{sss:Prf:setup}.

We consider in turn every small sphere $S\in\Sph$ containing a curve
$\gamma\in\CC$ on its boundary. There is then another small sphere
$S'\in\Sph$ also containing $\gamma$ on its boundary. For $\widehat S$
and $\widehat{S'}$, these boundary points appear as cusps in the
scaled hyperbolic metrics that were assigned to them
in~\S\ref{sss:AdjAtCusps}. We truncate the cusps on
$\widehat S,\widehat {S'}$ along horocycles and attach
$\widehat S,\widehat {S'}$ through an almost flat hyperboloid as
described in~\ref{sss:Plumbing}. The hyperboloid has a curvature
$\approx-\epsilon^2$, perimeter $\approx2\pi w_\gamma$ and length
$\approx\ell_\gamma$.

We have, in this manner, constructed a metric sphere $X\simeq (S^2,A)$
by plumbing together truncated small spheres in $\Sph$.  For every
$S\in \Sph$ we denote by $S^\circ$ the image of $\widehat S$ in
$X$. We also denote by $\Ann$ the set of almost flat annuli. We stress
that $\Ann$ is in bijection with $\CC$.

Suppose $B\in \Ann$ is an annulus connecting small spheres $S^\circ_1$
and $S^\circ_2$. Let $B_1$ be the subannulus of $B$ consisting of
points in $B$ that are closer to $S_1^\circ$ than $S_2^\circ$. Since
$B_1$ is constructed by enlarging the metric in
$\widehat S_1\setminus S^\circ_1$ we can view
$B\hookrightarrow \widehat S_1\setminus S^\circ_1$; we will refer to
this map as \emph{natural}. By construction,
\begin{lem}\label{lem:ExpNatMap}
  The natural map $B\to \widehat S_1\setminus S^\circ_1$ is
  contracting.  \qed
\end{lem}

\subsubsection{Dynamics at small spheres}
Recall that $X$ consists of truncated small spheres and of almost flat
cylinders connecting truncated spheres.

Consider first a small sphere $S\in \Sph$ and its $f$-image $S'$. We
have a rational map $f_S\coloneqq f\colon\widehat S\to
\widehat{S'}$. For all points in $S^{\circ}$ with $f_S$ image in
$S'^{\circ}$ we set $F$ to be $f_S$. The remaining points are bounded
by $f_S^{-1}(\partial S'^\circ)$. We now extend $F$ to $S^\circ$.

Consider a curve $\gamma\in f_S^{-1}(\partial S'^\circ)$. Then either
$\gamma$ is non-essential rel $A$ or $\gamma \in \CC$ rel $A$. In the first
case $\gamma$ bounds a peripheral disc $U$ containing at most one
point in $A$. By construction, see~\S\ref{sss:GlobMetr}, there is a
very long almost flat annulus $B\in \Ann$ attached to $F(\gamma)$.
Since $f_S\restrict U$ is expanding, we may extend $F$ to $U$, see
Lemma~\ref{lem:ExpNatMap}, in such a manner that $F\restrict U$ is
expanding. If there is an $a\in A\cap U$, then we require that
$F(a)=f(a)$ and that $F$ maps a neighbourhood of $a$ analytically
(i.e.~locally conformal except at $a$ where the map needs not be an
isomorphism) to a neighbourhood of $f(a)$.

Suppose that $\gamma$ is isotopic to a curve, say $\gamma_2$, in
$\partial S^\circ$. Denote bu $U$ the annulus between $\gamma$ and
$\gamma_2$. Let $B\in \Ann$ be the almost flat annulus attached to
$F(\gamma)$. We define $F$ on $U$ to be the composition of $f_S$ with
the inverse of the natural map from $B$ to
$\widehat S\setminus S^\circ$. In this manner we construct an
expanding extension of $F$ to $U$, see Lemma~\ref{lem:ExpNatMap}.

\subsubsection{Dynamics at annuli}
So far $F$ is defined on small spheres; let us assume that
$F\restrict S =f\restrict S$ for every $S\in \Sph$. We now extend $F$ to
$X\simeq (S^2,A)$ in an expanding manner so that $F\simeq f$.

Consider an annulus $B\in \Ann$. Suppose that $f$ maps $B$ to a
sequence of annuli and spheres $B_1,S_1,B_2,S_2,\dots, B_t$ with
$B_i\in \Ann$ and $S_i\in \Sph$. Consider two cases.

Suppose first $t=1$. Then $\ell_{B}-\ell_{B_1}\ggcurly 1$ because the
values $\ell_{B}> \ell_{B_1}$ from Proposition~\ref{prop:CombGl} are
rescaled so that $\ell_{B},\ell_{B_1}\ggcurly 1$. Also, either
$w_{B}>w_{B_1}$ or $w_{B}>w_{B_1}/2$ in case $f\restrict{B}$ has degree
greater than $1$. Therefore, we can map in an expanding manner $B$ to
$B_1$ minus a small (i.e.~of scale $\llcurly \ell$) neighbourhood of
$\partial B_1$ (which is already in the image of small spheres) in an
expanding manner so that the obtained map $F$ is isotopic to $f$ rel
$\partial B$. Indeed, identify $B$ and
$B_1\setminus (\text{small neighbourhood of }\partial B_1)$ with
$\mathbb S^1\times [0,1]$, recalling that $B,B_1$ are almost
flat. Then set $F$ to be $(x,y) \to (d x +m y, y )$, where $d\ge 1$ is
the degree of $f\restrict B$ and $m \ge 0$ is the twisting parameter. Since
$m$, $d$ are independent of $\ell \ggcurly 1 \ggcurly w$, the map
$F\restrict B$ is expanding.

Suppose next $t>1$. Subdivide $B$ into
$B'_1,S'_1,B'_2,\dots,S'_{t-1},B'_t$ so that each $S_i$ is an annulus
of length $\approx w$ and each $B'_j$ is an annulus of length
$\approx \ell/t$.  Again, since $\ell \ggcurly 1 \ggcurly w$ we can
define $F\restrict B \simeq f\restrict B$ in such a manner that $F$
expands $S'_i$ and $B'_j$ into $S_i$ and $B_j$ respectively.

\subsubsection{Perturbation of the metric}
We have constructed a metric space $(X,\mu)$ and a map
$F\colon X\selfmap$ which weakly ($\ge$) expands the metric, and such
that an iterate of $F$ is expanding.

\begin{lem}\label{lem:PertMetric}
  There is a small perturbation $\mu'$ of $\mu$ such that
  $F\colon X\selfmap$ expands $\mu'$.
\end{lem}
\begin{proof}
  Let $F^p\colon X\selfmap $ be an expanding iteration of $F$. By
  construction, $\mu$ is a smooth Riemannian metric such that $F$ is
  conformal (rel $\mu$) in a small neighbourhood of $A$.

  Denote by $A^\infty$ the set of periodic critical cycles of
  $F\restrict A\selfmap$. Recall that $A^\infty$ is the set of points
  at infinite distance from $X\setminus A^\infty$ for $\mu$. We also
  recall that cone points of $\mu$ belong to $A\setminus A^\infty$.

  For $i\le p-1$ consider the pulled-back metric
  $\mu_i\coloneqq (F^{-i})^*\mu$. Then $\mu_i$ is a Riemannian metric
  with cones in $f^{-i}(A\setminus A^\infty)$ and singularities in
  $f^{-i}(A^\infty)$. Moreover, $F$ weakly expands $\mu_i$.

  Write $\mu_i(z)$ as a conformal metric $\sigma_i(z) |\dd z|$ for
  $z\in X$ written in complex charts. For a sufficiently large $K>1$
  the inequality $\sigma_i(z)>K$ holds only in a small neighbourhood of
  $f^{-i} (A)$. Let $A^p\supset A^\infty$ be the set of periodic
  points in $A$. For sufficiently large $K$ and for $z$ close to
  $f^{-i}(A)\setminus A^p$ we define
  $\bar \sigma_i(z)\approx\min\{ \sigma_i(z), K\}$ so that $F$ still
  weakly expands the truncated metric
  $\bar \mu_i(z)= \bar \sigma_i(z) |\dd z|$. We leave $\sigma_i$
  unchanged away from the neighborhood of $A^p$.

  We claim that for a sufficiently small $\varepsilon>0$ the quadratic
  form
  \[\mu'\coloneqq \mu+(\bar\mu_1 +\dots + \bar\mu_{p-1})\varepsilon\]
  is positive define (i.e.~$\mu'$ is a metric) and that $F$ expands
  $\mu'$. Indeed, away from $A^p$ all $\bar \mu_{i}$ are finite
  metrics. Therefore, if $\varepsilon$ is sufficiently small, then
  $\mu'$ is positive definite away from $A^p$; so $\mu'$ is a
  metric. Since $F$ is conformal in a small neighbourhood of $A^p$ all
  $\bar \mu_i$ and $\mu$ are conformal metrics in a common
  charts. Hence $\mu'$ is positive-definite as a sum of conformal
  metrics.

  Since $F^p$ is expanding, $F$ expands at least one of
  $\mu, \bar\mu_1,\dots,\bar\mu_{p-1}$. Therefore, $F$ expands $\mu'$.
 \end{proof}

\subsection{Isotopy of expanding maps}
Let $f,g\colon (S^2,A)\selfmap$ be two expanding maps. Denote by
$\Fatou(f)$ and $\Fatou(g)$ the Fatou sets of $f$ and $g$
respectively. We may partially order the maps $f,g$ by declaring that
$g$ is ``smaller than'' $f$ if $A\cap\Fatou(g)\subset
A\cap\Fatou(f)$. In this sense, small maps are more expanding, and
B\"ottcher maps are maximal.

\begin{lem}\label{lem:ConjBetwExpMaps}
  Let $f,g\colon (S^2,A)\selfmap$ be two expanding maps with
  $A\cap\Fatou(f)=A\cap\Fatou(g)$. Then $f$ and $g$ are conjugate by
  $h\simeq \one$ if and only if $f\simeq g$.
\end{lem}
Moreover, if $\#A\ge 3$, then $h$ is unique,
see~\cite{bartholdi-dudko:bc3}*{\S\ref{bc3:ss:ghost}}.
\begin{proof}
  We show that if $f, g$ are isotopic, then they are conjugate by
  $h\simeq \one$. This is an application of the pullback argument.

  Choose $h_0,h_1\simeq \one$ such that $h_1f = g h_0$. We adjust $h_0$
  so that it respects B\"ottcher coordinates around periodic points in
  $A\cap\Fatou(f)$. Thus $h_0$ is equal to $h_1$ in a small
  neighbourhood of $A\cap\Fatou(f)$.

  Let inductively $h_n$ be the lift of $h_{n-1}$; i.e.\
  $h_n f = g h_{n-1}$. By construction, all $h_n$ coincide in a small
  neighbourhood of $A\cap\Fatou(f)$.

  Since $f$ is expanding away from $A\cap\Fatou(f)$, the
  sequence $h_n$ tends to a continuous map
  $h_\infty\colon (S^2,A)\selfmap$ satisfying
  $h_\infty f = g h_\infty$.

  Observe now that we also have $h^{-1}_n g = f h^{-1}_{n-1}$. Since
  $g$ is expanding away from $A\cap\Fatou(g)$, the sequence $h_n^{-1}$
  tends to a continuous map $h'_\infty\colon (S^2,A)\selfmap$
  satisfying $h'_\infty g = f h'_\infty$. Clearly,
  $h'_\infty h_\infty=\one$; i.e.\ $h_\infty$ is a homeomorphism.
\end{proof}

For a Thurston map $f\colon (S^2,A)\selfmap$, a \emph{Levy arc} is a
non-trivial path, with (possibly equal) starting and ending point in
$A$, that is isotopic rel $A$ to one of its iterated lifts.  Let $A'$
be a forward-invariant subset of $A$. We say that $A'$ is
\emph{homotopically isolated} if there is no Levy arc connecting two
points in $A'$.

\begin{lem}\label{lem:HomIsolCond}
  Suppose that $f\colon (S^2,A)\selfmap$ is a B\"ottcher expanding
  map, that $A'\subset A\cap \Fatou(f)$ is forward invariant, and that
  $\Fatou'$ is the set of points in $\Fatou(f)$ attracted by
  $A'$. Then $A'$ is homotopically isolated if and only if the
  following properties hold:
  \begin{enumerate}
  \item if $O$ is a connected component of $\Fatou'$, then
    $\overline O$ is a closed topological disc and, moreover,
    $A\cap\partial O=\emptyset$;
  \item if $O_1,O_2$ are different connected components of
    $\Fatou'$, then $\overline O_1\cap \overline O_2=\emptyset$.
  \end{enumerate}
\end{lem}
\begin{proof}
  Suppose first that $A'$ is not homotopically isolated. Let $\ell$ be
  a Levy arc connecting points $a,b\in A'$. Then $\ell$ can be
  realized as an inner ray $R_1$ followed by an inner ray $R_2$. If
  $a\not= b$, then the closures of the Fatou components centered at
  $a$ and $b$ intersect. If $a=b$ but $R_1\not=R_2$, then the closure
  of the Fatou component centered at $a$ is not a closed disc, since
  it is pinched at $a=b$. If $R_1=R_2$, then the landing point of
  $R_1$ belongs to $A$.

  Conversely, let us assume that $A'$ is homotopically isolated. We
  first verify that $A\cap\partial O=\emptyset$. Indeed, if
  $a\in A\cap\partial O$, then the internal ray $R$ of $O$ landing at
  $a$ is preperiodic. For $n$ large enough, the ray $f^n(R)$ is a
  periodic ray of $f^n(O)$ connecting its center, which is a point in
  $A'$, to $f^n(a)\in A$. Therefore, a loop starting at the center of
  $f^n(O)$, then following $f^n(R)$, then circling $f^n(a)$, and then
  following $f^n(R)$ back to the center of $f^n(O)$ is a Levy arc.

  If the conclusion of the lemma does not hold, then either there is a
  periodic component $O$ of $\Fatou'$ which is not a disk, and then
  there are two different inner rays $R_1,R_2$ of $O$ that land
  together; or there are two periodic connected components $O_1,O_2$
  of $\Fatou'$ and respective inner rays $R_1\subset O_1$ and
  $R_2\subset O_2$ that land together.

  If $R_1,R_2$ are inner rays of $O$ that land together, then we have
  $f^n(R_1)\neq f^n(R_2) $ for all $n\ge 0$. Indeed, otherwise the
  common landing point of $R_1,R_2$ would be precritical,
  contradicting $A\cap\partial O=\emptyset$. Furthermore, for all
  $n$ sufficiently large $f^n(R_1)\cup f^n(R_2) $ is a closed curve,
  non-null-homotopic rel $A$. Indeed, if $f^n(R_1)\cup f^n(R_2) $ were
  trivial for some $n$, then $f^m(R_1)\cup f^m(R_2)$ would be trivial
  for all $m\in \{0,1,\dots, n\}$. Let then $D_m$ be the open disc
  bounded by $f^m(R_1)\cup f^m(R_2) $ and not intersecting $A$.  We
  see that $f^m\colon D_0\to D_m $ has degree one. Denote by $\phi_m$
  the angle in $D_m$ between $f^m(R_1)$ and $f^m(R_2)$ measured at the
  center $f^m(a)$ of $f^m(O)$. Then $\phi_m=\deg_a(f^m)\phi_0$. Since
  $\phi_0>0$ because $R_1\neq R_2$, and $\deg_a(f^m)\to\infty$ as
  $m\to\infty$ because $O$ is a Fatou component, we see that
  $f^{n}\colon D_0\to D_n$ has degree greater than one for all
  sufficiently large $n$.

  In all cases, we obtain for some $n>m \ge0$ an arc
  $f^n(R_1)\cup f^n(R_2)$ that is isotopic to $f^m(R_1)\cup f^m(R_2)$
  $A$, so $f^n(R_1)\cup f^n(R_2)$ is a Levy arc.
\end{proof}

Suppose that $\sim$ is a closed equivalence relation on $S^2$ whose
equivalence classes are connected and filled-in (namely, with
connected complement) compact subsets of $S^2$ and suppose that not
all points of $S^2$ are equivalent. In this case Moore's
theorem~\cite{moore:sphere} states that the quotient space $S^2/{\sim}$
is homeomorphic to $S^2$.

\begin{cor}
  Suppose that $f\colon (S^2,A)\selfmap$ is a B\"ottcher expanding map
  and suppose that $A'\subset A\cap \Fatou(f)$ is a forward invariant
  homotopically isolated subset of $A$. Let $\Fatou'$ be the set of
  points in $ \Fatou(f)$ attracted by $A'$. Then the equivalence
  relation $\sim$ on $S^2$ specified by
  \[x\sim y\Longleftrightarrow \begin{cases}
      x=y \;\text{ or }\\
      x,y\text{ are in the closure of the same connected component of
      }\Fatou';
    \end{cases}
  \]
  is an $f$-invariant equivalence relation satisfying Moore's
  theorem. View $(S^2,A)/{\sim} \simeq (S^2,A)$. The induced map
  $f/{\sim}\colon (S^2,A)/{\sim} \selfmap $ is topologically expanding and is isotopic
  rel $A$ to $f$.
\end{cor}
\begin{proof}
  It is clear that $f/{\sim}$ is topologically expanding. If we view
  $(S^2,A)/{\sim} \simeq (S^2,A)$, then $f$ and $f/{\sim}$ have isomorphic
  bisets; therefore $f\simeq f/{\sim}$.
\end{proof}

\begin{prop}\label{prop:semiconjugate}
  Let $f,g\colon (S^2,A)\selfmap$ be two expanding maps such that
  $f\simeq g$ and $A\cap\Fatou(g)\subseteq A\cap\Fatou(f)$. Write
  $A'\coloneqq A\cap (\Fatou(f)\setminus \Fatou(g))$ and let $\Fatou'$
  be the set of points attracted towards $A'$ under iteration of $f$.

  Then there is a semiconjugacy $\pi\colon (S^2,A)\to (S^2,A)$ from
  $f$ to $g$ defined by
  \[\pi(x)=\pi(y)\Longleftrightarrow  \begin{cases}
 x=y \;\text{ or }\\
 x,y\text{ are in the closure of the same connected component of
    }\Fatou';
\end{cases}.\]
\end{prop}
As in Lemma~\ref{lem:ConjBetwExpMaps}, the semiconjugacy $\pi$ is
unique.
\begin{proof}
  It sufficient to prove this proposition for the case when $f$ is a
  B\"ottcher expanding map.  By Lemma~\ref{lem:HomIsolCond} applied to
  $g$ we see that $A_g$ is homotopically isolated. Therefore, again by
  Lemma~\ref{lem:HomIsolCond} we can collapse $\Fatou'$ to obtain a
  topologically expanding map $f/\Fatou'$. Since $f/\Fatou'\approx g$,
  the claim now follows from Lemma~\ref{lem:ConjBetwExpMaps}.
\end{proof}

\section{Computability of the Levy decomposition}\label{ss:algo}
In this section, we give algorithms that prove
Corollaries~\ref{cor:decidegeom} and~\ref{cor:decidelevy}. 

Recall that a branched covering $f\colon (S^2,P_f,\ord_f)\selfmap$ is
doubly covered by a torus endomorphism if and only if $P_f$ contains
exactly four points and $\ord_f(P_f)=\{2\}$. Moreover, in this case
$f\colon (S^2,P_f,\ord_f)\selfmap$ is itself an orbifold self-covering
and its biset $B(f)$ is right principal. It is easy to see that
$G:=\pi_1(S^2,P_f,\ord_f)$ is isomorphic to
$\Z^2 \rtimes_{-\one} \Z/2$, and that $B(f)$ is of the following form:
for a $2\times 2$ integer matrix $M$ with $\det(M)>1$ and a vector
$v\in\Z^2$ denote by $M^v\colon \Z^2 \rtimes_{-\one} \Z/2 \selfmap$
the endomorphism given by a ``cross product structure''
(see~\cite{bartholdi-dudko:bc0}*{Proposition~\ref{bc0:prop:bis:2222}}):
\begin{equation}\label{eq:InjEndOfK:bc4}
  M^{v}(n,0)=(Mn,0) \text{ and }M^{v}(n,1)=(Mn+v,1).
\end{equation}  
Then $B(f)$ is isomorphic to $G$ as a set, with left and right actions
given by $g\cdot b\cdot h=M^v(g)b h$ for all $g,b,h\in G$.  Moreover,
$f\colon (S^2,P_f,\ord_f)\selfmap$ is combinatorially equivalent to
the quotient of $z\mapsto M z+v\colon \R^2/\Z^2\selfmap$ by the
involution $z\mapsto-z$.  Indeed, every endomorphism of $G$ is of the
form~\eqref{eq:InjEndOfK:bc4}.

\begin{algo}\label{algo:is2cover}
  \textsc{Given} a sphere biset ${}_G B_G$,\\
  \textsc{Decide} whether $B$ is the biset of a map double covered by
  a torus endomorphism \textsc{as follows:}\\\upshape
  \begin{enumerate}
  \item Compute the action of $B$ on peripheral conjugacy classes in $G$.
  \item Determine the minimal orbisphere structure $(S^2,\ord_B)$ from
    the action on peripheral conjugacy classes,
    see~\S\ref{ss:orbisphere}.
  \item Return \texttt{yes} if the Euler characteristic of
    $(S^2,\ord_B)$ is $=0$ and $\#B=4$, and \texttt{no} otherwise.
  \end{enumerate}
\end{algo}

\begin{algo}\label{algo:getparam}
  \textsc{Given} a sphere biset ${}_G B_G$ of a map double covered by
  a torus endomorphism,\\
  \textsc{Compute} parameters $M,v$ for the torus endomorphism
  $z\mapsto M z+v$ \textsc{as follows:}\\\upshape
  \begin{enumerate}
  \item As in Algorithm~\ref{algo:is2cover}, compute the action of $B$
    on peripheral conjugacy classes in $G$, and determine the quotient
    map $\pi\colon G\to\overline G$ to the minimal orbisphere
    structure, see~\S\ref{ss:orbisphere}, and the quotient biset
    ${}_{\overline G}\overline B_{\overline G}$.
  \item Note that $\overline G$ is of the form $\Z^2\rtimes\Z/2$,
    where the $\Z^2$ is generated by all even products of peripheral
    generators and the $\Z/2$ is generated by any chosen generator.
  \item Since the map corresponding to $B$ is a covering, the biset
    $\overline B$ is left-free and right-principal. Choose an
    arbitrary element $\overline x\in\overline B$, thus identifying
    $\overline B$ with $\overline G$ via
    $\overline x g\leftrightarrow g$.
  \item Let $\{g_0,g_1\}$ be a basis of $\Z^2\subset\overline G$, and
    choose a peripheral generator $h$ of $\overline G$. Write
    $g_0\overline x=\overline x g_0^a g_1^b$ and
    $g_1\overline x=\overline x g_0^c g_1^d$ for some $a,b,c,d\in\Z$
    which form the matrix
    $M=(\begin{smallmatrix}a&c\\ b&d\end{smallmatrix})$, and write
    $h\overline x=\overline x g_0^e g_1^f h$ for some $e,f\in\Z$
    forming the vector $v=(\begin{smallmatrix}e\\ f\end{smallmatrix})$.
  \end{enumerate}
\end{algo}

The following algorithm determines whether a biset is \Tor.  We shall
give, in~\cite{bartholdi-dudko:bc3}, a much more efficient encoding of
non-post-critical marked periodic points, and improve the speed of
Algorithm~\ref{algo:istor}. The present algorithm relies on the following
\begin{thm}[\cite{selinger-yampolsky:geometrization}*{Main Theorem~II}]\label{thm:nikita}
  Let $f$ be a Thurston map that is doubly covered by a torus endomorphism. If $f$ is Levy-free, then it is \Tor.
\end{thm}

\begin{algo}\label{algo:istor}
  \textsc{Given} a sphere biset ${}_G B_G$ of a map double covered by
  a torus endomorphism,\\
  \textsc{Decide} whether $B$ is the biset of a \Tor\ map \textsc{as
    follows:}\\\upshape
  \begin{enumerate}
  \item Use Algorithm~\ref{algo:getparam} to obtain a $2\times2$
    matrix $M$ expressing the linear part of the endomorphism covering
    $B$, and return \texttt{no} if $M$ has $\pm1$ as eigenvalue.
  \item Choose a basis $X$ of $B$. Using the action of $B$ on
    peripheral conjugacy classes, determine those (call them $A'$)
    that correspond to non-post-critical points.
  \item Make the finite list of all choices $\widehat{A'}$ of periodic
    points or preperiodic points on the torus that map to each other
    as the peripheral conjugacy classes map to each other under $B_*$.
  \item Run the following two steps in parallel. By
    Theorem~\ref{thm:nikita}, precisely one of them will terminate:
  \item For an enumeration of all multicurves $\CC$, check whether
    $\CC$ is a Levy cycle, and if so return \texttt{no}.
  \item For each choice $\widehat{A'}$ of periodic points, compute the
    biset $\widehat{B(A')}$ of the map $(z\mapsto M z+v)/\{\pm1\}$
    with $(\frac12\Z^2/\Z^2\cup\widehat{A'})/\{\pm1\}$ marked, and go
    through the countably many maps $X\to\widehat{B(A')}$. If one of
    these maps extends to an isomorphism of bisets, return
    \texttt{yes}.
  \end{enumerate}
\end{algo}

\begin{algo}\label{algo:isexp}
  \textsc{Given} a sphere biset ${}_G B_G$,\\
  \textsc{Decide} whether $B$ is the biset of an expanding map
  \textsc{as follows:}\\\upshape
  \begin{enumerate}
  \item Check, using Algorithm~\ref{algo:is2cover}, whether $B$ is
    double covered by a torus endomorphism. If not, run the next two
    steps in parallel. If $B$ is double covered by a torus
    endomorphism, then run Algorithm~\ref{algo:getparam} to obtain a
    $2\times2$ matrix $M$ expressing the linear part of the
    endomorphism, and run Algorithm~\ref{algo:istor} to decide whether
    ${}_G B_G$ is geometric biset. If ${}_G B_G$ is not a geometric
    biset or at least one eigenvalue of $M$ has absolute value less
    than $1$, then return \texttt{no}. Otherwise return \texttt{yes}.
  \item Enumerate all finite subsets of $G$, and check whether one is
    the nucleus of $(B,X)$. If so, return \texttt{yes}.
  \item Simultaneously, enumerate all multicurves $\CC$ on $(S^2,A)$,
    and check whether any is a Levy obstruction for $B$. If so, return
    \texttt{no}.
  \end{enumerate}
  By Theorem~\ref{thm:main}, either Step (2) or Step (3) will succeed.
\end{algo}

The following algorithm computes the Levy decomposition, and proves in
this manner Corollary~\ref{cor:decidelevy}:
\begin{algo}\label{algo:levy}
  \textsc{Given} a Thurston map $f\colon(S^2,A)\selfmap$ by its biset,\\
  \textsc{Compute} the Levy decomposition of $f$ \textsc{as follows:}\\\upshape
  \begin{enumerate}\setcounter{enumi}{-1}
  \item We are given a $G$-$G$-biset $B=B(f)$. Recall that multicurves
    on $(S^2,A)$ are treated as collections of conjugacy classes in
    $G$. Their $B$-lift is computable
    by~\cite{bartholdi-dudko:bc1}*{\S\ref{bc1:ss:ccgroups}}.
  \item For an enumeration of all multicurves $\CC$ on $(S^2,A)$, that
    never reaches a multicurve before reaching its proper
    submulticurves, do the following steps:
  \item If the multicurve $\CC$ is not invariant, or is not Levy,
    continue in~(1) with the next multicurve.
  \item Compute the decomposition of $B$ using the algorithm
    in~\cite{bartholdi-dudko:bc2}*{Theorem~\ref{bc2:thm:DecompOfBiset}};
  \item If all return bisets of the decomposition are either of degree
    $1$, or expanding (recognized using Algorithm~\ref{algo:isexp}),
    or \Tor\ (recognized using Algorithm~\ref{algo:istor}), then
    return $\CC$;
  \item Proceed with the next multicurve.
  \end{enumerate}
\end{algo}

\section{Amalgams}\label{ss:matings}
In the previous sections, we considered a single Thurston map --- or,
equivalently, sphere biset --- and characterized when it is
combinatorially equivalent to an expanding map.

In this section, we rather consider a Thurston map that is defined as
an ``amalgam'' of small maps, glued together along a multicurve; we
derive a criterion for the amalgam to be expanding. A typical example
is a \emph{formal mating}, which is a sphere map admitting an
``equator'' --- a simple closed curve $\gamma$ isotopic to its lift,
which maps back to $\gamma$ by maximal degree. We first give an
algebraic characterization in terms of bisets, and then its geometric
translation in terms of internal rays.

\subsection{Sphere trees of bisets}
We briefly recall
from~\cite{bartholdi-dudko:bc2}*{Definition~\ref{bc2:defn:gfofbisets}}
the notion of \emph{sphere tree of biset}: firstly, we are given a
tree $\gf$ of groups, namely a tree with a group attached to every
vertex and edge, and inclusions $G_e\to G_{e^-}$ and isomorphisms
$G_e\leftrightarrow G_{\overline e}$ from an edge $e$ respectively to
its source $e^-$ and its reverse $\overline e$. Secondly, we are given
analogously a tree $\gfB$ of bisets, and two graph morphisms
$\lambda,\rho\colon\gfB\to\gf$, such that $\rho$ is a graph
covering and $\lambda$ is monotonous (preimages of connected sets are
connected).

The graph of groups $\gf$ has a \emph{fundamental group}
$\pi_1(\gf,*)$ at each vertex $*\in\gf$; this is the group of
expressions of the form
$(g_0,e_0,g_1,\dots,e_{n-1},g_n)$ with
$(e_0,\dots,e_{n-1})$ a closed path in $\gf$ based at $*$ and
$g_i\in G_{e_i^-}$, subject to natural relations coming from the
edge group inclusions. Likewise, the graph of bisets $\gfB$ has
a \emph{fundamental biset}, which is an ordinary biset for the
fundamental group. Just as sphere bisets (up to isomorphism) capture
Thurston maps (up to isotopy), sphere trees of bisets capture Thurston
maps with an invariant multicurve.

Consider a sphere group $G$ and a sphere $G$-biset $B$. A \emph{Levy
  cycle} in $B$ is a periodic sequence of conjugacy classes
$g_0^G,\dots,g_{m-1}^G,g_m^G=g_0^G$ such that each $g_i^G$ is a
$B$-lift of $g_{i+1}^G$; namely, there are biset elements
$b_0,\dots,b_{m-1}\in B$ such that $g_i b_i=b_i g_{i+1}$ holds for all
$i=0,\dots,m-1$. More succinctly, in the product biset $B^{\otimes m}$,
we have the commutation relation $g_0 b=b g_0$.
\begin{lem}
  Let $f\colon(S^2,A)\selfmap$ be a Thurston map not doubly covered by
  a torus endomorphism map. Then $f$ admits a Levy cycle if and only
  if $B(f)$ admits one.
\end{lem}
\begin{proof}
  If $(g_0^G,\dots,g_{m-1}^G)$ is a Levy cycle in $B(f)$, then $B(f)$
  is not contracting, so $f$ is not expanding by
  Theorem~\ref{thm:main}, so $f$ contains a Levy cycle again by
  Theorem~\ref{thm:main}.

  Conversely, let $(\gamma_0,\dots,\gamma_{m-1})$ be a Levy cycle for
  $f$, and write each $\gamma_i$ as a conjugacy class $g_i^G$. Since
  each $\gamma_{i+1}$ has an $f$-lift isotopic to $\gamma_i$, there
  are biset elements $b_0,\dots,b_{m-1}$ such that
  $g_i^{\pm G}b_i\ni b_i g_{i+1}$. Up to replacing some $g_i$ by their
  inverses, we may assume $g_i^G b_i\ni b_i g_{i+1}$ except possibly
  $g_{m-1}^G b_{m-1}\ni b_{m-1}g_0$. In that case, increase $m$ to
  $2m$ and set $g_{m+i}=g_i^{-1}$ for $i=0,\dots,m-1$ so as to have
  $g_i^G b_i\ni b_i g_{i+1}$ for all $i$, namely
  $g_i^{h_i}b_i=b_i g_{i+1}$ for some elements $h_i\in G$. Set finally
  $c_i\coloneqq h_i b_i$ to obtain $g_i c_i=c_i g_{i+1}$ for all
  $i$. Thus $(g_0^G,\dots,g_{m-1}^G)$ is a Levy cycle in $B$.
\end{proof}

The following definition captures the notion of algebraic Levy cycles
for graphs of bisets:
\begin{defn}\label{defn:rpc}
  Let $\gfB$ be a sphere tree of bisets.  A \emph{periodic
    pinching cycle} for $\gfB$ is
  \begin{enumerate}
  \item a sequence of $m$ closed paths
    $\gamma_j\coloneqq(v_{0,j},e_{1,j},v_{1,j},\dots,e_{n,j},v_{n,j}=v_{0,j})$
    in the tree $\gfB$, for $j=0,\dots,m-1$, such that
    $\rho(\gamma_{j+1})=\lambda(\gamma_j)$, indices read modulo $m$;
  \item a sequence of $m\times n$ biset elements
    $b_{i,j}\in B_{e_{i,j}}$ and group elements
    $g_{i,j}\in G_{\rho(v_{i,j})}$, for $i=0,\dots,n-1$ and
    $j=0,\dots,m-1$, satisfying
    \[g_{i,j+1} b_{i+1,j}^-=b_{i,j}^+ g_{i,j}\text{ for all $i,j$},\]
    indices being read cyclically.\qedhere
  \end{enumerate}
\end{defn}
Consider a periodic pinching cycle. Note that the elements
$g_{0,j}\rho(e_{i,j})g_{1,j}\cdots\rho(e_{n,j})$, for $j=0,\dots,m-1$,
define elements of the fundamental group of $\gf$ based at
$\rho(v_{0,j})$, and that their conjugacy classes again produce a Levy
cycle for the fundamental biset of $\gfB$.

We shall always assume that periodic pinching cycles are non-trivial:
$m,n>0$ and the elements
$g_{0,j}\rho(e_{i,j})g_{1,j}\cdots\rho(e_{n,j})$ are reduced in the
fundamental group of $\gf$.

Recall also that, in a tree of bisets $\gfB$, vertices of
$\gfB$ are classified as \emph{essential} and
\emph{inessential}; every vertex $v\in\gf$ has a unique
$\lambda$-preimage $\iota(v)\in\gfB$ that is
essential. Consider a vertex $v\in\gf$, and assume that
$(\rho\circ\iota)^m(v)=v$ for some $m>0$. The corresponding
\emph{return biset} is
$B_{\iota(v)}\otimes\cdots B_{\iota(\rho\circ\iota)^{m-1}(v)}$, and is
a $G_v$-biset. We denote by $R(\gfB)$ the set of all return
bisets of $\gfB$.

Let $\gf$ be a tree of sphere groups with fundamental group
$G=\pi_1(\gf,*)$. Recall that the edge groups $G_e$ in $\gf$ embed as
cyclic subgroups of $G$. Choose a generator $t_e\in G_e$ for every
edge $e\in\gf$, and consider the collection of their conjugacy classes
$\CC=\{t_e^G\mid e\in E(\gf)\}$.  We call $\CC$ the \emph{edge
  multicurve} of $\gf$.

Given a sphere biset $B$, recall that its \emph{portrait} is the
induced map $B_*\colon A\selfmap$ on the set of peripheral conjugacy
classes. A portrait is \emph{hyperbolic} if every periodic cycle of
$B_*$ contains a critical peripheral class; i.e.~if $B$ is the biset
of a rational map $f$, then all critical points of $f$ are in the
Fatou set.

\begin{thm}\label{thm:algebraic amalgam}
  Let $\gfB$ be a sphere tree of bisets, and let
  $B\coloneqq\pi_1(\gfB)$ denote its fundamental biset. Assume that
  the portrait of $B$ is hyperbolic. Then $B$ is sphere contracting if
  and only if the following all hold:
  \begin{enumerate}
  \item All return bisets in $R(\gfB)$ are contracting;
  \item The edge multicurve of $\gfB$ contains no Levy cycle;
  \item There is no non-trivial periodic pinching cycle for $\gfB$.
  \end{enumerate}
\end{thm}
\begin{proof}
  Each of the conditions is clearly necessary: if a return biset of
  $\gfB$ is not contracting, then its image in $B$ is still not
  contracting; if an edge multicurve is a Levy cycle then it is a Levy
  cycle for $B$; and, by definition, a periodic pinching cycle has an
  iterated lift that is isotopic to itself, so every periodic pinching
  cycle generates a Levy obstruction.

  Conversely, assume that every return biset in $\gfB$ is contracting,
  that the edge multicurve $\CC$ of $\gfB$ is Levy-free, and that $B$
  is not contracting. Then by Theorem~\ref{thm:main} there is a Levy
  cycle in $B$. Write $G=\pi_1(\gf,*)$, and let
  $\{\ell_0^G,\dots,\ell_{m-1}^G\}$ denote this Levy cycle.  The
  conjugacy classes $\ell_j^G$ are not reduced to conjugacy classes in
  vertex or edge groups, because return bisets are contracting and the
  edge multicurve is Levy-free, so every $\ell_j^G$ admits a
  representative $\ell_j$ of the form
  $g_{0,j}f_{1,j} g_{1,j}\dots f_{n(j),j}\in \pi_1(\gf,w_j)$; here $f_{1,j}\dots f_{n(j),j}$ is a loop in $\gf$ based
  at $w_j$, and $g_{i,j}\in G_{f_{i,j}^+}$. Furthermore, if we require
  each $n(j)$ to be minimal, then this expression of a representative
  is unique up to cyclic permutation.

  Since $\{\ell_0^G,\dots,\ell_{m-1}^G\}$ is a Levy cycle, there are
  $b_0,\dots, b_{m-1}\in B$ with $\ell_j b_j=b_j \ell_{j+1}$ for all
  $j$. Furthermore, since the tree of bisets $\gfB$ is left fibrant,
  every $b_j\in B$ may be written as $b_j = h_j c_j$ for some
  $c_j\in B_{v_{0,j+1}}$ the vertex biset of a vertex $v_{0,j}\in\gfB$
  with $\rho(v_{0,j})=w_j$, and some element
  $h_j\in\pi_1(\gf,w_j,w_{j+1})$ in the path groupoid of $\gf$. We get
  \[\ell_j^{h_j} c_j = c_j \ell_{j+1}\text{ for all }j=0,\dots,m-1.\]

  Now, again because $\gfB$ is left fibrant, each path $\ell_j$ lifts
  by $\rho$ to a unique path
  $\gamma_j\coloneqq(v_{0,j},e_{1,j},v_{1,j},\dots,e_{n(j),j},v_{n(j),j}=v_{0,j})$,
  and the above equation gives
  $\lambda(\gamma_{j+1})=\ell_j^{h_j}$. In particular, the length of
  $\ell_j^{h_j}$ is at most the length of $\ell_{j+1}$; it follows
  that all $\ell_j^{h_j}$ are cyclically reduced, and all have the
  same length $n$.

  We may now redefine $\ell_j$ as the appropriate cyclic permutation
  of itself so that $\ell_j c_j=c_j\ell_{j+1}$ holds for all
  $j=0,\dots,m-2$, and we have
  $\ell_{m-1}^{h_{m-1}}c_{m-1}=c_{m-1}\ell_0$, where
  $\ell_{m-1}^{h_{m-1}}$ is a cyclic permutation of $\ell_{m-1}$. At
  worst replacing $m$ by $m n$ and letting $\ell_{km+j}$ be the
  appropriate cyclic permutation of $\ell_j$ for all $j=0,\dots,m-1$
  and all $k=0,\dots,n-1$, we may ensure that
  $\ell_j c_j=c_j\ell_{j+1}$ holds for all $j$. Set
  $c_{0,j}\coloneqq c_j$ and choose $c_{i,j}\in B_{f_{i,j}}$ so that
  $g_{i,j+1} c_{i+1,j}^-=c_{i,j}^+ g_{i,j}$ holds. We have constructed
  a periodic pinching cycle.
\end{proof}

Furthermore, it is decidable whether $\gfB$ admits a periodic pinching
cycle: for example, Algorithm~\ref{algo:isexp} tells us whether the
fundamental biset $B$ is expanding; in that case, there is no periodic
pinching cycle, while if not then a periodic pinching cycle may be
found by enumerating all $m n$-tuples of biset and group elements as in
Definition~\ref{defn:rpc}.

\subsection{Trees of correspondences}
The algebraic construction above is closely related to the topological
construction of an ``amalgam'' $\mathfrak F$ of maps.  We shall not
stress too precisely the conditions that must be satisfied by
$\mathfrak F$, but rather give an intuitive connection to the previous
subsection: on the one hand, such a formalism is well developed
in~\cite{pilgrim:combinations}; on the other hand, the algebraic
picture is the one that we use in practice.

We may start with the following data: firstly, one is given a finite
tree $\mathfrak T$ expressing a decomposition of a marked sphere
$(S^2,A)$. Let there be a topological sphere $S_v$ for every vertex
$v\in\mathfrak T$, and a cylinder (written $S_e$) for every edge
$e\in\mathfrak T$. There is a finite set $A_v\subset S_v$ of marked
points assigned to each vertex $v\in \mathfrak T$. If whenever $e$
touches $v$ one removes a small disk around a certain marked point
from $S_v$ and attaches its boundary to a boundary of the cylinder
$S_e$, one obtains after gluing a marked sphere $(S^2,A)$ so that $A$
is $\bigcup_v A_v\setminus\{\text{removed points}\}$.

Secondly, one is given a tree of correspondences: a tree $\mathfrak F$
also expressing a decomposition of a marked sphere and two graph
morphisms $\lambda,\rho\colon\mathfrak F\to\mathfrak T$. To every
vertex and edge $z\in\mathfrak F$ one is given a ``topological
correspondence'' between the spaces $\lambda(z)$ and $\rho(z)$. More
precisely, for each vertex $v\in \mathfrak F$ one is given a marked
sphere $(S_v,A_v)$, a covering map
$S_v\setminus A_v \to S_{\rho(v)}\setminus A_{\rho(v)}$, and an
inclusion $S_v\to S_{\lambda(v)}$ (note that $\lambda(v)$ needs not be
a vertex). Similarly, for every edge $e\in \mathfrak F$ one is given a
cylinder $S_e$ together with a covering map $S_e\to S_{\rho(e)}$ and
an inclusion $S_e\to S_{\lambda(e)}$.  The marked set $A$ is assumed
to be forward invariant and contains all critical values of all
correspondences $F_z$.

Typical examples to consider are matings (as we saw in the
introduction), for which the trees $\mathfrak T$ and $\mathfrak F$
have a singe edge; the correspondence at each vertex $v_\pm$ is the
polynomial $p_\pm$, and the correspondence at the edge is
$z\mapsto z^d$ if the cylinder is modelled on $\C^*$.

We denote by $R(\mathfrak F)$ the \emph{small maps} of $\mathfrak F$,
namely the return maps to vertex spheres obtained by composing the
correspondences along cycles. Again in the example of matings,
the small maps are $p_\pm$.

By the ``van Kampen theorem'' for bisets,
see~\cite{bartholdi-dudko:bc1}
and~\cite{bartholdi-dudko:bc2}*{Theorem~\ref{bc2:thm:pilgrim}}, we may
freely move between the languages of trees of correspondences
$\mathfrak F$, sphere trees of bisets $\gfB$, sphere bisets with
invariant algebraic multicurve $(B,\CC)$ represented as conjugacy
classes in the fundamental group, and Thurston maps with invariant
multicurve $f\colon(S^2,A,\CC)\selfmap$. We call $f$ the \emph{limit}
of $\mathfrak F$.

Let $\mathfrak F$ be a tree of correspondences with B\"ottcher
expanding return maps. Let $\CC$ denote the invariant multicurve
associated with the edges of $\mathfrak F$, namely $\CC$ is the set of
core curves of cylinders represented by edges of $\mathfrak T$. We
assume that $\CC$ is Levy-free. Let $\CC_0$ denote the union of
primitive unicycles in $\CC$. Consider $\gamma\in \CC_0$ and denote by
$S_e$ the cylinder with core curve $\gamma$, for $e\in \mathfrak
T$. Since $\gamma$ is contained in a primitive unicycle, there is a
unique $f\in \mathfrak F$ with $\lambda(f)=e$. We call the core curve
of $S_{\rho (f)}$ the \emph{image} of $\gamma$. In this manner, there
is a well defined (up to isotopy) first return map
$f_\gamma\colon \gamma\to \gamma$; up to isotopy we assume that
$f_\gamma$ is conjugate to $z\to z^d\colon S^1\selfmap$, with $d>1$
because $\CC$ is Levy-free.

The curve $\gamma$ is on the boundary of two small periodic spheres,
call them $S_1$ and $S_2$. By assumption, the first return maps on
$S_1$ and $S_2$ are B\"ottcher expanding. There are periodic Fatou
components $F_1\subset S_1$ and $F_2\subset S_2$ such that $\gamma$ is
viewed as the circle at infinity of $F_1$ an $F_2$. Then points in
$\gamma$ parametrize internal rays of $F_1$ and $F_2$, and periodic
internal rays are parameterized by periodic points of
$f_\gamma\colon \gamma\selfmap$, namely by rationals of the form
$m/(d^n-1)$ for some $m,n\in\N$.
 
\begin{defn}
\label{defn:topPinchCycle}
  Let $\mathfrak F$ be a tree of correspondences with expanding return
  maps. Let $\CC$ denote the invariant multicurve associated with
  the edges of $\mathfrak T$.  Let $\CC_0$ denote the union of the
  primitive unicycles in $\CC$.

  A \emph{periodic pinching cycle} for $\mathfrak F$ is a sequence
  $z_1,\dots,z_n$ of periodic points on $\CC_0$, and a sequence of
  internal rays $I_1^\pm,\dots,I_n^\pm$ in the Fatou components of
  small maps in $\mathfrak F$ touching $\CC_0$, such that, indices read
  modulo $n$,
  \begin{itemize}
  \item $I_i^+$ and $I_{i+1}^-$ are both parameterized by $z_i$, and
    lie in neighbouring spheres;
  \item $I_i^+$ and $I_i^-$ both land at the same point and in the
    same sphere.\qedhere
  \end{itemize}
\end{defn}

As mentioned above, topological periodic pinching cycles are the form
that Levy cycles take in trees of correspondences with expanding
return maps: given a Levy cycle, we may put it in minimal position
with respect to the multicurve $\CC$ associated with the edges of the
tree, and thus decompose the Levy cycle into periodic arcs, with arc
contained in a small sphere and connecting two boundary circles.

If we choose basepoints on the small spheres and boundary circles, and
paths from the boundary circle basepoint to the neighbouring sphere
basepoints, we may translate these arcs into loops in fundamental
groups of small spheres.

Even though it is not necessary for our argument, let us explain more
precisely how to construct an algebraic periodic pinching cycle out of
a topological one. For simplicity assume that all small spheres and
cylinders are fixed. Choose basepoints $*_t$ on all small spheres and
curves $S_t$ in $\mathfrak T$, identifying the group $G_t$ with
$\pi_1(S_t,*_t)$ and the biset $B_t$ with homotopy classes of paths
from $*_t$ to an $f$-preimage of $*_t$. Choose for each edge
$e\in\mathfrak T$ a path $\ell_e$ from $*_e$ to $*_{e^-}$.

Consider a periodic pinching cycle for $\mathfrak F$, and assume again
for simplicity that all rays $I_i^\pm$ are fixed. To every fixed point
$z_i$, say $z_i\in S_{t(i)}$, corresponds a biset element
$b_i\in B_{t(i)}$: choose a path $\gamma_i$ in $S_{t(i)}$ from
$*_{t(i)}$ to $z_i$, and set
$b_i\coloneqq\gamma_i\#f^{-1}(\gamma_i^{-1})$. Since $f$ is expanding,
the infinite concatenation of lifts $b_i^\infty$ is a path from
$*_{t(i)}$ to $z_i$. Note that we are using, here, the identification
of the circle $S_{t(i)}$ with the Julia set of $z^d$ for some $d>1$
and with the Julia set $\Julia(B_{t(i)})$ of the biset $B_{t(i)}$;
recall from~\S\ref{ss:limit} that it consists of equivalence classes
of bounded (here constant $b_i^\infty$) infinite sequences in
$B_{t(i)}$. Let the rays $I_i^\pm$ belong to sphere $S_{v(i)}$, and
set
$g_i\coloneqq\ell_{t(i-1)}^{-1}\#b_{i-1}^\infty\#I_i^-\#(I_i^+)^{-1}\#(b_i^\infty)^{-1}\#\ell_{t(i)}\in
G_{v(i)}$. Then these data $b_i,g_i,v(i),t(i)$ determine an algebraic
periodic pinching cycle with $m=1$. In general, the periodic pinching
cycle for $\mathfrak F$ will be periodic but not fixed, and $m$ will
be $>1$.

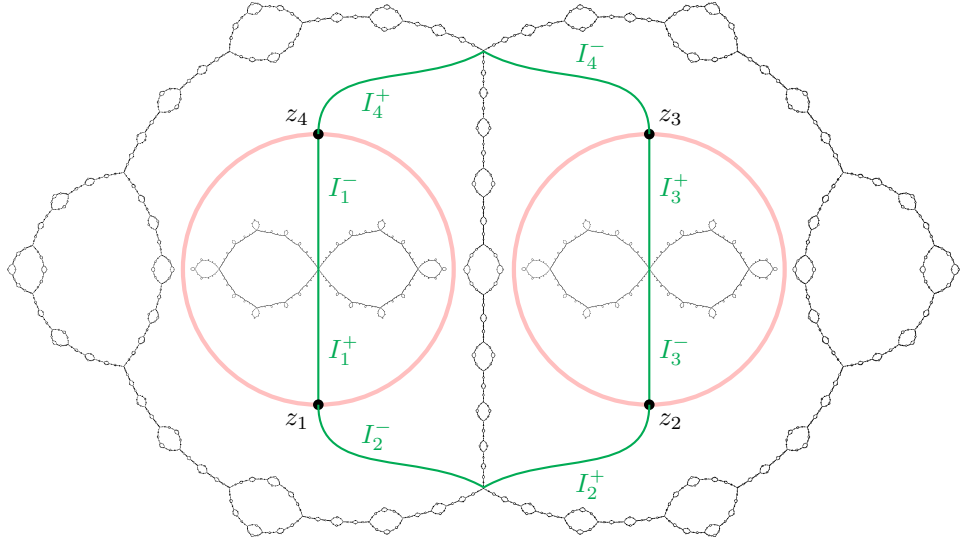
\begin{figure}
\centering
\begin{tikzpicture}
  \node at (0,0) {\includegraphics[width=\textwidth]{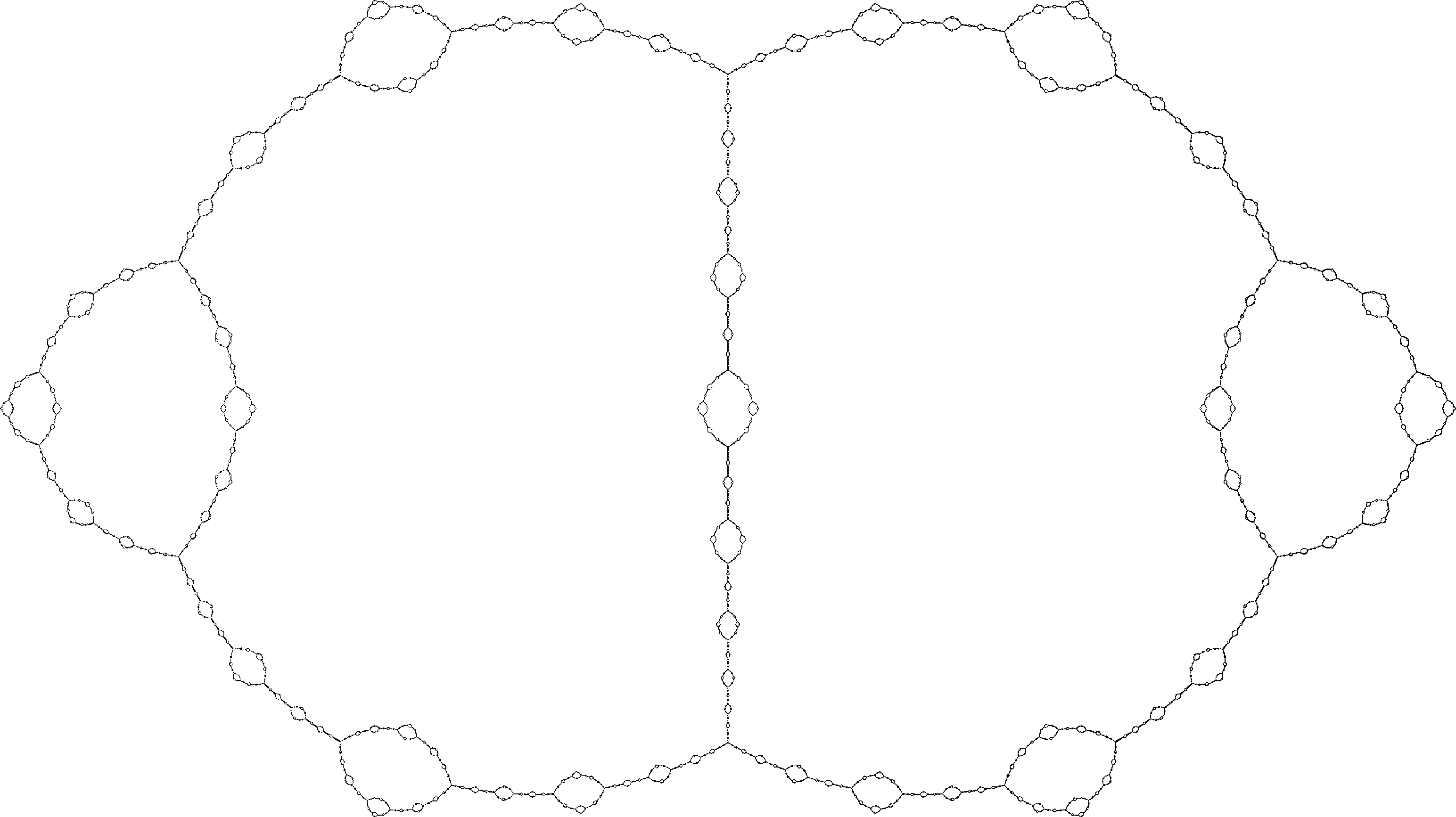}};
  \foreach\x/\l/\m/\n in {-2.2/1/+/-,2.2/3/-/+} {
    \node at (\x,0) {\includegraphics[width=34mm]{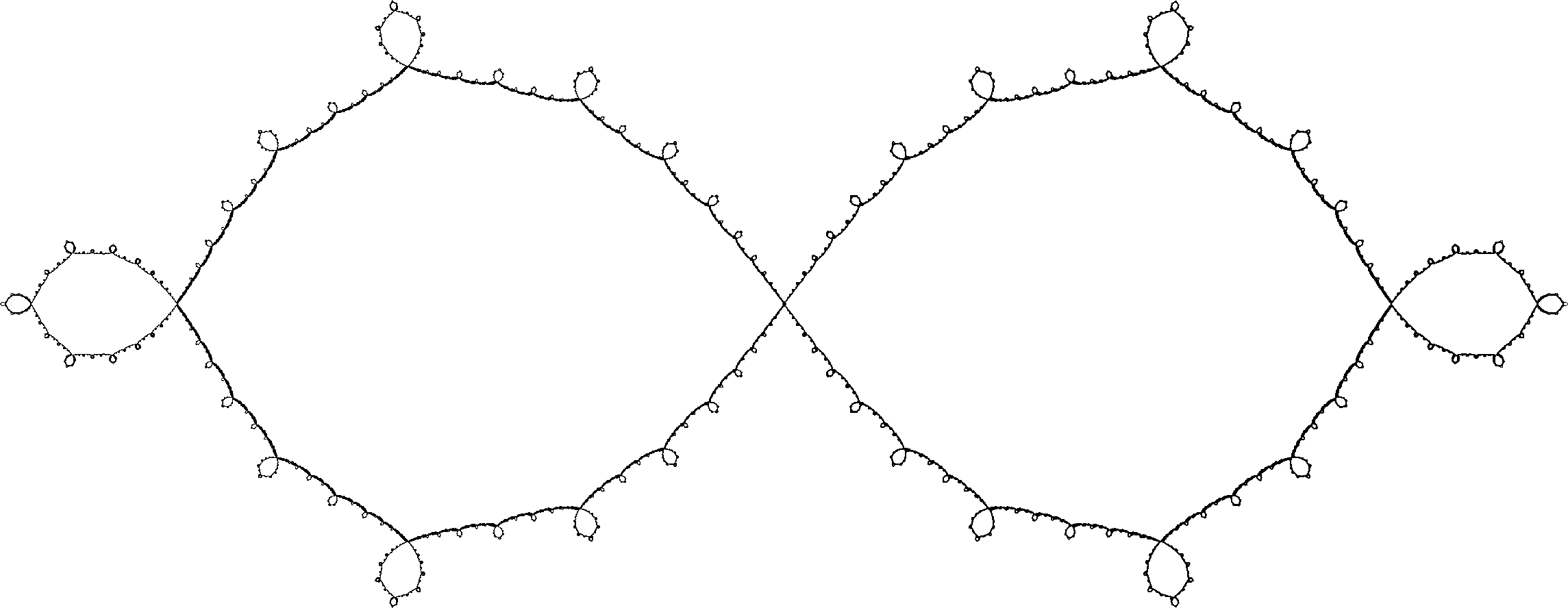}};
    \draw[red!25,ultra thick] (\x,0) circle (18mm);
    \draw[green!67!blue,thick] (\x,0) -- node[right,pos=0.6] {$I_\l^\m$} (\x,-1.8);
    \draw[green!67!blue,thick] (\x,0) -- node[right,pos=0.6] {$I_\l^\n$} (\x,1.8);
  };
  \fill (-2.2,-1.8) circle (2pt) node [below left] {$z_1$};
  \fill (2.2,-1.8) circle (2pt) node [below right] {$z_2$};
  \fill (2.2,1.8) circle (2pt) node [above right] {$z_3$};
  \fill (-2.2,1.8) circle (2pt) node [above left] {$z_4$};
  \foreach\x/\y/\Y/\p/\l/\m in {+/+/-/above/4/-,+/-/+/below/2/+,-/+/-/below/4/+,-/-/+/above/2/-} {
    \draw[green!67!blue,thick] (\x2.2,\y1.8) .. controls +(0.0,\y1.0) and +(\x0.866,\Y0.5) .. node[\p] {$I_\l^\m$} (0,\y2.9);
  }
\end{tikzpicture}
\caption{A periodic pinching cycle. There is a central fixed sphere
  mapping under $z^3\frac{2z-1}{2-z}$, and two spheres attached on the
  Fatou components of $0$ and $1$ mapping under $3z^2-2z^2$. The
  periodic pinching cycle is in green, and the edges of the tree of
  spheres are in red.}
\end{figure}

Given a Thurston map $f\colon(S^2,A)\selfmap$, recall that its
\emph{portrait} is the induced map $f\colon A\selfmap$ with its local
degree. A portrait is \emph{hyperbolic} if all its cycles contain a
point of degree $>1$.

\begin{thm}\label{thm:expamalgam}
  Let $\mathfrak F$ be a tree of maps with hyperbolic portraits. Then
  its limit $f\colon(S^2,A)\selfmap$ is isotopic to an expanding map
  if and only if the following all hold:
  \begin{enumerate}
  \item All small maps of $\mathfrak F$ are isotopic to expanding
    maps;
  \item The invariant multicurve associated with the edges of
    $\mathfrak T$ is Levy-free;
  \item There is no non-trivial periodic pinching cycle for
    $\mathfrak F$.
  \end{enumerate}
\end{thm}
\begin{proof}
  This is a direct translation of Theorem~\ref{thm:algebraic
    amalgam}. It it instructive to give a geometric proof of the only
  non-trivial implication, namely that if $f$ admits a Levy cycle $L$
  then it admits a periodic pinching cycle.

  Put $L$ in minimal position with respect to $\CC$. By
  Proposition~\ref{prop:solidPerCycl}\eqref{prop:solidPerCycl:2}, a
  Levy cycle may only intersect a primitive unicycle. Choose a curve
  $\ell\in L$, and let $z_1,\dots,z_n$ denote, in cyclic order along
  $\ell$, the intersections of $\ell$ with $\CC$. Assuming that all
  small maps are expanding, the pieces of $\ell$ between points $z_i$
  and $z_{i+1}$ belong to Fatou components and their boundaries, and
  may be assumed to be internal rays. We have in this manner obtained
  a periodic pinching cycle.
\end{proof}

\subsection{Higher-degree matings}
We are ready to prove Theorem~\ref{thm:mating}. Note that, in the
case of matings, periodic pinching cycles of periodic angles are
precisely the periodic pinching cycles defined above for amalgams.

\subsubsection{Polynomials}\label{ss:polynomials}
Let $f$ be a complex polynomial of degree $d\ge2$. The \emph{filled-in
  Julia set} $\mathcal K(f)$ of $f$ is
\[\mathcal K(f)=\{z\in\C\mid f^n(z)\not\to\infty\text{ as }n\to\infty\}.\]
Assume that $\mathcal K(f)$ is connected, and let $\phi$ be the
inverse of the B\"ottcher coordinate associated with the Fatou
component of $\infty$, so we have
$\phi\colon\hC\setminus \mathcal K(f)\to\hC\setminus\overline{\mathbb
  D}$ satisfying $\phi(f(z))=\phi(z)^d$ and $\phi(\infty)=\infty$ and
$\phi'(\infty)=1$. For $\theta\in\R/\Z$, the associated \emph{external
  ray} $R_\theta$ is defined as
$\{\phi^{-1}(r e^{2i\pi\theta})\mid r>1\}$.

We have $\Julia(f)=\partial\mathcal K(f)$. Assume now that $f$ is post-critically finite; in particular $\Julia(f)$ is locally connected. Then the landing point
$\pi(\theta)\coloneqq\lim_{r\to1^+}\phi_f^{-1}(re^{2i\pi\theta})$ of
the ray $R_f(\theta)$ exists for all $\theta$, and defines a
continuous map $\pi\colon\R/\Z\to\Julia(f)$.

On the other hand, consider a basepoint $*\in\C\setminus A$ very close
to $\infty$, so that its preimages $*_0,\dots,*_{d-1}$ are all also
very close to $\infty$. Let $t\in\pi_1(\C,*)$ denote a short
counterclockwise loop around $\infty$, and choose for all
$i=0,\dots,d-1$ a path $\ell_i$ from $*$ to $*_i$ that remains in the
neighbourhood of $\infty$, and in such a manner that the paths
$\ell_i\#f^{-1}(t)$ and $\ell_{i+1}$ are homotopic for all
$i=0,\dots,d-2$, and $\ell_{d-1}\#f^{-1}(t)$ is homotopic to
$t\#\ell_0$. Here by `$s\#f^{-1}(t)$' we denote the concatenation of a
path $s$ with the unique $f$-lift of $t$ that starts where $s$ ends.

The following proposition illustrates the link between Julia sets (see
also~\S\ref{ss:fatou}) and bisets in the concrete case of polynomials.

\begin{prop}
\label{prop:JuliaEncod}
  The set $X\coloneqq\{\ell_0,\dots,\ell_{d-1}\}$ is a basis of $B(f)$. Let
  $\rho\colon\{0,\dots,d-1\}^\infty\to\R/\Z$ be the base-$d$ encoding
  map $x_1x_2\dots\mapsto\sum x_i d^{-i}$; then the following diagram
  commutes:
  \[\begin{tikzcd}
      X^\infty\ar[r,equal]\ar[dd,swap,"/{\sim}"] & \{0,\dots,d-1\}^\infty\ar[d,"\rho"]\\
      & \R/\Z\ar[d,"\pi"]\\
      \Julia(B(f))\ar[r,<->] & \Julia(f)
    \end{tikzcd}
  \]
  where $\sim$ is the asymptotic equivalence relation defined
  in~\S\ref{ss:limit}.
\end{prop}
\begin{proof}
  Consider $x_1 x_2\dots\in X^\infty$ with each $x_i=\ell_{m_i}$ for
  some $m_i\in\{0,\dots,d-1\}$. Then the path
  $x_1\#f^{-1}(x_2)\#f^{-2}(x_3)\cdots$ is a well-defined path in
  $\C\setminus \mathcal K(f)$, which has a limit because $f$ is
  expanding, and has the same limit as $R_f(\theta)$ for
  $\theta=\rho(m_1m_2\dots)$ because with respect to the hyperbolic metric of ${\C}\setminus K(f)$ there is a $\delta>0$ such that $x_1\#f^{-1}(x_2)\#f^{-2}(x_3)\cdots$ is in the~$\delta$-neighborhood of $R_f(\theta)$.
\end{proof}

\begin{proof}[Proof of Theorem~\ref{thm:mating}]
  $\eqref{thm:mating:1}\Rightarrow\eqref{thm:mating:2}$: assume that
  $p_+\FM p_-\colon\mathbb S\selfmap$ is combinatorially equivalent to
  an expanding map $h\colon S^2\selfmap$. Denote by $\Sigma$ the quotient
  of $\mathbb S$ in which all external rays are shrunk to points.

  Let $\Julia_\pm$ denote the Julia set of $p_\pm$ respectively, and
  denote their common image in $\Sigma$ by $\Julia$. We have a
  well-defined map $p_+\GM p_-\colon\Julia\selfmap$, and we shall see
  that it is conjugate to $h\colon\Julia(h)\selfmap$. Let
  $\pi_\pm(\theta)$ denote the landing point of the external ray with
  angle $\theta$ on $\Julia_\pm$.

  Fix a basepoint of $*$ at infinity, and choose a set $X$ of paths
  from $*$ to all its $(p_+\FM p_-)$-preimages on the circle at
  infinity; the cardinality of $X$ is the common degree of $p_+$ and
  $p_-$. The bisets $B(p_+)$ and $B(p_-)$ may be chosen to have the
  same basis $X$ consisting of these paths. Note that the basis $X$ is
  the standard one for $p_+$, but is reversed for $p_-$.  Let their
  respective nuclei be $N_\pm$. Denoting by $\sim_\pm$ the
  corresponding asymptotic equivalence relations we have, according
  to~\S\ref{ss:limit}, conjugacies
  \[X^\infty/{\sim_+}\cong\Julia_+\text{ and }X^\infty/{\sim_-}\cong\Julia_-.
  \]
  The bisets $B(h)$ and $B(p_+\FM p_-)$ are isomorphic, and since $h$
  is expanding the nucleus of $B(h)$ is contained in
  $(N_+\cup N_-)^\ell$ for some $\ell\in\N$. It follows that the
  equivalence relation $\sim_h$ associated with the nucleus of $N(h)$
  is generated, as an equivalence relation, by
  ${\sim_+}\cup{\sim_-}$. By Proposition~\ref{prop:JuliaEncod} we therefore have
  \begin{align*}
    \Julia(h)\cong X^\infty/{\sim_h} &\cong\frac{(X^\infty/{\sim_+})\sqcup(X^\infty/{\sim_-})}{[w]_{\sim_+}=[w]_{\sim_-}\text{ for all }w\in X^\infty}\\
    &\cong\frac{\Julia_+\sqcup\Julia_-}{\pi_+(\theta)=\pi_-(-\theta)\text{ for all }\theta\in S^1}\cong\Julia\subseteq \Sigma,
  \end{align*}
  conjugacies between the dynamics of $h$, $p_+\FM p_-$ and
  $p_+\GM p_-$.

  We then extend this conjugacy between the Julia sets of $h$ and
  $p_+\GM p_-$ to Fatou components, which are all discs. The critical
  portraits of $p_+\FM p_-$ and of $p_+\GM p_-$ coincide, so their
  periodic Fatou components are in natural bijection. Since every
  Fatou component is ultimately periodic, we extend the bijection by
  pulling back by $p_+\FM p_-$ and $p_+\GM p_-$ respectively. The
  bijection between the Julia sets restricts to bijections between
  boundaries of Fatou components, which are ultimately periodic
  embedded circles in the Julia sets; this extends uniquely the
  bijection between Julia sets to a conjugacy
  $(S^2,h)\to(\Sigma,p_+\GM p_-)$.

  $\eqref{thm:mating:2}\Rightarrow\eqref{thm:mating:3}$ is clear,
  because a pinching cycle is made of external rays, so it shrinks to
  a node in $X$, and therefore $X$ is not a
  topological sphere.

  $\eqref{thm:mating:3}\Rightarrow\eqref{thm:mating:1}$ is
  Theorem~\ref{thm:expamalgam}.
\end{proof}

We remark that the criterion due to Mary Rees and Tan Lei gives strong
constraints on pinching cycles of periodic angles in degree
$2$. Firstly, the associated external rays must land at dividing fixed
points. Secondly, in Definition~\ref{defn:topPinchCycle} it may be
assumed that $n=2$, namely each curve in a pinching cycle intersects
the equator in exactly two points.  This is not true anymore in higher
degree; here is an example in degree $3$.
\begin{exple}\label{exple:degree3}
  Consider the polynomials $q_\pm=\frac12z^3\pm\frac32z$. The
  polynomial $q_+$ has two fixed critical points at $\pm i$, and $q_-$
  exchanges its two critical points at $\pm1$.

  Let $p_+$ be the tuning of $q_+$ in which the local map $z^2$ is
  replaced by the Basilica map $z^2-1$ on the immediate basins of
  $\pm i$, and let $p_-$ be the tuning of $q_-$ in which the return
  map $z^2\circ z^2$ on the immediate basin of $1$ is replaced by
  $(1-z)^2+1\circ z^2$.  Then $p_\pm$ are polynomials of degree $3$,
  with $4$ finite post-critical points forming $2$ periodic
  $2$-cycles. The supporting rays for $p_+$ are
  $\{\{1/8,11/24\},\{5/8,23/24\}\}$ and those for $p_-$ are
  $\{\{1/8,19/24\},\{5/8,7/24\}\}$; the maps are
  $\approx z^3\pm 2.12132z$.

  The only periodic external rays landing together for $q_+$ are at
  angles $0$ and $1/2$, while the only periodic external rays landing
  together for $q_-$ are at angles $1/4$ and $3/4$. It follows that
  the only pairs of external rays landing together for $p_+$ and $p_-$
  are
  \begin{xalignat*}{2}
    R_{p_+}(1/8),& R_{p_+}(3/8) & R_{p_-}(1/8), & R_{p_-}(7/8)\\
    R_{p_+}(0),& R_{p_+}(1/2) & R_{p_-}(1/4), & R_{p_-}(3/4)\\
    R_{p_+}(5/8),& R_{p_+}(7/8) & R_{p_-}(3/8), & R_{p_-}(5/8)
  \end{xalignat*}
  
  It then follows that the sequence of rays $R_{p_+}(1/8)$,
  $R_{p_+}(3/8)$, $R_{p_-}(3/8)$, $R_{p_-}(5/8)$, $R_{p_+}(5/8)$,
  $R_{p_+}(7/8)$, $R_{p_-}(7/8)$, $R_{p_-}(1/8)$ is a periodic
  pinching cycle, so $p_+ \FM p_-$ is not equivalent to an expanding
  map. On the other hand, there does not exist any periodic pinching
  cycle with $n=2$.
\end{exple}


\bib{bartholdi-h-n:aglt}{article}{
  author={Bartholdi, Laurent},
  author={Henriques, Andr\'e G.},
  author={Nekrashevych, Volodymyr V.},
  title={Automata, groups, limit spaces and tilings},
  date={2006},
  journal={J. Algebra},
  volume={305},
  number={2},
  pages={629\ndash 663},
  review={\MR {2266846 (2008a:20070)}},
  doi={10.1016/j.jalgebra.2005.10.022},
  eprint={arXiv:math/0412373},
}

\bib{bartholdi-dudko:bc0}{article}{
  author={Bartholdi, Laurent},
  author={Dudko, Dzmitry},
  title={Algorithmic aspects of branched coverings},
  date={2015},
  eprint={arXiv:cs/1512.05948},
  status={submitted},
}

\bib{bartholdi-dudko:bc1}{article}{
  author={Bartholdi, Laurent},
  author={Dudko, Dzmitry},
  title={Algorithmic aspects of branched coverings I. Van Kampen's theorem for bisets},
  date={2015},
  eprint={arXiv:cs/1512.08539},
  status={submitted},
}

\bib{bartholdi-dudko:bc2}{article}{
  author={Bartholdi, Laurent},
  author={Dudko, Dzmitry},
  title={Algorithmic aspects of branched coverings II. Sphere bisets and their decompositions},
  date={2016},
  eprint={arXiv:math/1603.04059},
  status={submitted},
}

\bib{bartholdi-dudko:bc3}{unpublished}{
  author={Bartholdi, Laurent},
  author={Dudko, Dzmitry},
  title={Algorithmic aspects of branched coverings III. Erasing maps and orbispaces},
  date={2015},
  status={in preparation},
}

\bib{bonk-meyer:etm}{article}{
  author={Bonk, Mario},
  author={Meyer, Daniel},
  title={Expanding Thurston maps},
  date={2010},
  eprint={arXiv:1009.3647},
}

\bib{buff+:questionsaboutmatings}{article}{
  author={Buff, Xavier},
  author={Epstein, Adam L.},
  author={Koch, Sarah},
  author={Meyer, Daniel},
  author={Pilgrim, Kevin},
  author={Rees, Mary},
  author={Lei, Tan},
  title={Questions about polynomial matings},
  language={English, with English and French summaries},
  journal={Ann. Fac. Sci. Toulouse Math. (6)},
  volume={21},
  date={2012},
  number={5},
  pages={1149--1176},
  issn={0240-2963},
  review={\MR {3088270}},
  doi={10.5802/afst.1365},
}

\bib{douady-h:thurston}{article}{
  author={Douady, Adrien},
  author={Hubbard, John~H.},
  title={A proof of Thurston's topological characterization of rational functions},
  journal={Acta Math.},
  volume={171},
  date={1993},
  number={2},
  pages={263\ndash 297},
  issn={0001-5962},
  review={\MR {1251582 (94j:58143)}},
}

\bib{farb-margalit:mcg}{book}{
  author={Farb, Benson},
  author={Margalit, Dan},
  title={A primer on mapping class groups},
  series={Princeton Mathematical Series},
  volume={49},
  publisher={Princeton University Press},
  place={Princeton, NJ},
  date={2012},
  pages={xiv+472},
  isbn={978-0-691-14794-9},
  review={\MR {2850125}},
}

\bib{griffiths-harris:pag}{book}{
  author={Griffiths, Phillip},
  author={Harris, Joseph},
  title={Principles of algebraic geometry},
  note={Pure and Applied Mathematics},
  publisher={Wiley-Interscience [John Wiley \& Sons], New York},
  date={1978},
  pages={xii+813},
  isbn={0-471-32792-1},
  review={\MR {507725 (80b:14001)}},
}

\bib{haissinsky-pilgrim:cxc}{article}{
  author={Ha{\"{\i }}ssinsky, Peter},
  author={Pilgrim, Kevin~M.},
  title={Coarse expanding conformal dynamics},
  language={English, with English and French summaries},
  journal={Ast\'erisque},
  number={325},
  date={2009},
  pages={viii+139 pp. (2010)},
  issn={0303-1179},
  isbn={978-2-85629-266-2},
  review={\MR {2662902 (2011k:37076)}},
}

\bib{haissinsky-pilgrim:algebraic}{article}{
  author={Ha{\"{\i }}ssinsky, Peter},
  author={Pilgrim, Kevin M.},
  title={An algebraic characterization of expanding Thurston maps},
  journal={J. Mod. Dyn.},
  volume={6},
  date={2012},
  number={4},
  pages={451--476},
  issn={1930-5311},
  review={\MR {3008406}},
}

\bib{harvey:curvecomplex}{article}{
  author={Harvey, W. J.},
  title={Boundary structure of the modular group},
  conference={ title={Riemann surfaces and related topics: Proceedings of the 1978 Stony Brook Conference}, address={State Univ. New York, Stony Brook, N.Y.}, date={1978}, },
  book={ series={Ann. of Math. Stud.}, volume={97}, publisher={Princeton Univ. Press, Princeton, N.J.}, },
  date={1981},
  pages={245--251},
  review={\MR {624817}},
}

\bib{kameyama:thurston}{article}{
  author={Kameyama, Atsushi},
  title={The Thurston equivalence for postcritically finite branched coverings},
  journal={Osaka J. Math.},
  volume={38},
  date={2001},
  number={3},
  pages={565\ndash 610},
  issn={0030-6126},
  review={\MR {1860841 (2002h:57004)}},
}

\bib{moore:sphere}{article}{
  author={Moore, Robert L.},
  title={Concerning upper semi-continuous collections of continua},
  journal={Trans. Amer. Math. Soc.},
  volume={27},
  date={1925},
  number={4},
  pages={416--428},
  issn={0002-9947},
  review={\MR {1501320}},
  doi={10.2307/1989234},
}

\bib{nekrashevych:ssg}{book}{
  author={Nekrashevych, Volodymyr~V.},
  title={Self-similar groups},
  series={Mathematical Surveys and Monographs},
  volume={117},
  publisher={American Mathematical Society, Providence, RI},
  date={2005},
  pages={xii+231},
  isbn={0-8218-3831-8},
  review={\MR {2162164 (2006e:20047)}},
  doi={10.1090/surv/117},
}

\bib{pilgrim:cto}{article}{
  author={Pilgrim, Kevin~M.},
  title={Canonical Thurston obstructions},
  journal={Adv. Math.},
  volume={158},
  date={2001},
  number={2},
  pages={154\ndash 168},
  issn={0001-8708},
  review={\MR {1822682 (2001m:57004)}},
  doi={10.1006/aima.2000.1971},
}

\bib{pilgrim:combinations}{book}{
  author={Pilgrim, Kevin~M.},
  title={Combinations of complex dynamical systems},
  series={Lecture Notes in Mathematics},
  volume={1827},
  publisher={Springer-Verlag},
  place={Berlin},
  date={2003},
  pages={x+118},
  isbn={3-540-20173-4},
  review={\MR {2020454 (2004m:37087)}},
}

\bib{selinger:augts}{article}{
  author={Selinger, Nikita},
  title={Thurston's pullback map on the augmented Teichm\"uller space and applications},
  journal={Invent. Math.},
  volume={189},
  date={2012},
  number={1},
  pages={111\ndash 142},
  issn={0020-9910},
  review={\MR {2929084}},
  doi={10.1007/s00222-011-0362-3},
}

\bib{selinger-yampolsky:geometrization}{article}{
  author={Selinger, Nikita},
  author={Yampolsky, Michael},
  title={Constructive geometrization of Thurston maps and decidability of Thurston equivalence},
  date={2013},
  eprint={arXiv:1310.1492},
}

\bib{shishikura-t:matings}{article}{
  author={Shishikura, Mitsuhiro},
  author={Lei, Tan},
  title={A family of cubic rational maps and matings of cubic polynomials},
  journal={Experiment. Math.},
  volume={9},
  date={2000},
  number={1},
  pages={29\ndash 53},
  issn={1058-6458},
  review={\MR {1758798 (2001c:37042)}},
}

\bib{tan:matings}{article}{
  author={Tan, Lei},
  title={Matings of quadratic polynomials},
  journal={Ergodic Theory Dynam. Systems},
  volume={12},
  year={1992},
  number={3},
  pages={589\ndash 620},
}

\bib{wittner:phd}{thesis}{
  author={Wittner, Ben},
  title={On the bifurcation loci of rational maps of degree two},
  type={Ph.D. Thesis},
  date={1988},
}



\begin{bibdiv}
\begin{biblist}
\bibselect{math}
\end{biblist}
\end{bibdiv}
\end{document}